\documentclass[11pt,reqno,a4paper]{amsart}
\usepackage{amssymb}
\usepackage{fancyhdr}
\usepackage{bbold}
\usepackage{ marvosym }\usepackage{stmaryrd}
\usepackage{amsmath}
\usepackage[utf8]{inputenc}
\usepackage[T1]{fontenc}

\usepackage[
    backend=biber,
    style=numeric,
    natbib=false,
    url=true, 
    doi=false,
    eprint=true,
    maxnames=50
]{biblatex}

\usepackage{mathrsfs}
\usepackage{ esint }
\usepackage[normalem]{ulem}
\usepackage{hyperref}\hypersetup{colorlinks=true, linkcolor=BrickRed, citecolor=RoyalBlue, urlcolor=Sepia}
\usepackage{enumitem}
\usepackage{mathtools}
\usepackage[dvipsnames]{xcolor}
\usepackage{ stmaryrd }
\usepackage{cleveref}
\usepackage{tikz}
\numberwithin{equation}{section}

\newtheoremstyle{mytheorem}
{10pt}% measure of space to leave above the theorem. E.g.: 3pt
{10pt}% measure of space to leave below the theorem. E.g.: 3pt
{\it}% name of font to use in the body of the theorem
{\parindent}% measure of space to indent
{\bf}% name of head font
{.}% punctuation between head and body
{ }% space after theorem head; " " = normal interword space
{\thmnumber{#2.~}\thmname{#1}\thmnote{~\rm#3}}% Manually specify head

\newtheoremstyle{myremark}
{10pt}% measure of space to leave above the theorem. E.g.: 3pt
{10pt}% measure of space to leave below the theorem. E.g.: 3pt
{\rm}% name of font to use in the body of the theorem
{\parindent}% measure of space to indent
{\bf}% name of head font
{}% punctuation between head and body
{ }% space after theorem head; " " = normal interword space
{\thmnumber{#2.~}\thmname{#1}\thmnote{~\rm#3}}% Manually specify head

\newtheoremstyle{myparagraph}
{10pt}% measure of space to leave above the theorem. E.g.: 3pt
{10pt}% measure of space to leave below the theorem. E.g.: 3pt
{\rm}% name of font to use in the body of the theorem
{\parindent}% measure of space to indent
{\bf}% name of head font
{.}% punctuation between head and body
{ }% space after theorem head; " " = normal interword space
{\thmnumber{#2.~}\thmname{#1}\thmnote{#3}}% Manually specify head

\theoremstyle{mytheorem}
\newtheorem{theorem}[subsubsection]{Theorem}
\newtheorem{definition}[subsubsection]{Definition}

\newtheorem{corollary}[subsubsection]{Corollary}
\newtheorem{proposition}[subsubsection]{Proposition}

\theoremstyle{myremark}
\newtheorem{remark}[subsubsection]{Remark.}

\theoremstyle{myparagraph}
\newtheorem{parag}[subsubsection]{}
\newtheorem*{parag*}{}

\makeatletter
    \def\@secnumfont{\sc}
    \def\section{\@startsection%
    {section}% name: section/subsection/etc.
    {1}% level: 1 for section/2 for subsection/etc.
    \z@{1.5\linespacing\@plus .2\linespacing}% vertical skip before
      {.7\linespacing}% vertical skip after
      {\normalfont\sc\centering}}% style
    \def\subsection{\@startsection{subsection}{2}
    \z@{1.0\linespacing\@plus .2\linespacing}% vertical skip before
      {.5\linespacing}% vertical skip after
      {\normalfont\it\centering}}% style
\makeatother
\makeatletter
\renewenvironment{proof}[1][\proofname]{\par 
    \pushQED{\qed}% 
    \normalfont \topsep10\p@\@plus6\p@\relax 
    \trivlist 
    \item[\hskip\labelsep 
    \bfseries 
    #1\@addpunct{.}]\ignorespaces}
    {\popQED\endtrivlist\@endpefalse} 
\providecommand{\proofname}{Proof}
\makeatother

\newlist{enumeraterem}{enumerate}{1}
\setlist[enumeraterem]{label=(\roman*), leftmargin=0pt, itemsep=3pt, itemindent=30pt}
\newlist{enumeratethm}{enumerate}{1}
\setlist[enumeratethm]{label={\rm(\roman*)}, leftmargin=30pt, itemsep=2pt}

\newcounter{cnstcnt}
\newcommand{\newr}{%
\refstepcounter{cnstcnt}%
\ensuremath{r_{\thecnstcnt}}}
\newcommand{\oldr}[1]{\ensuremath{r_{\ref{#1}}}}

\newcounter{const}

\newcommand{\Haus}{\mathscr{H}}
\newcommand{\Leb}{\mathscr{L}}
\newcommand{\Gr}{\mathrm{Gr}}
\newcommand{\gr}{\mathrm{gr}}
\newcommand{\Span}{\mathrm{span}}
\newcommand{\Tan}{{\rm Tan}}

\newcommand{\bd}{\partial}
\newcommand{\M}{\text{M}}
\newcommand{\Lip}{{\rm Lip}}

\DeclareMathOperator{\trace}{\mbox{\Large$\llcorner$}}
\DeclareMathOperator{\antitrace}{\mbox{\kern-1.5pt\Large$\lrcorner$\kern.5pt}}

\newcommand{\N}{\mathbb{N}}
\newcommand{\Z}{\mathbb{Z}}

\newcommand{\R}{\mathbb{R}}

\begin{document}

    % We do not use the maketitle command.
    % In the next lines we define the heading for the 
    % first page, then title, authors' names, and so on...
    % acknowledgements and affiliations are at the end of the paper

\thispagestyle{empty}
~\vskip -1.1 cm

	%
	% heading of first page
	%

\vspace{1.7 cm}

	%
	% title
	%
\begin{center}
{\Large\bf
Frobenius theorem
and fine structure of tangency sets to non-involutive distributions
}\end{center}

\vspace{.7 cm}

	%
	% authors' names
	%
{\centering\sc 
Giovanni Alberti, Annalisa Massaccesi, Andrea Merlo
\\
}

\vspace{.5 cm}

	%
	% abstract, keywords and MSC numbers
	%
{\rightskip 1 cm
\leftskip 1 cm
\parindent 0 pt
\footnotesize
{\sc Abstract.}
In this paper we provide a complete answer to the question whether Frobenius’ Theorem can be generalized to surfaces below the \(C^{1,1}\) threshold. We study the fine structure of the tangency set in terms of involutivity of a given distribution and we highlight a tradeoff behavior between the regularity of a tangent surface and that of the tangency set. First of all, we prove a Frobenius-type result, that is, given a $k$-dimensional surface $S$ of class $C^1$ and a non-involutive $k$-distribution $V$, if $E$ is a Borel set contained in the tangency set $\tau(S,V)$ of $S$ to $V$ and $\mathbb1_E\in W^{s,1}(S)$ with $s>1/2$ then $E$ must be $\Haus^k$-null in $S$. In addition, if \(S\) is locally a graph of a $C^1$ function with gradient in \(W^{\alpha,q}\) and if a Borel set \(E \subset \tau(S,V)\) satisfies \(\mathbb1_E \in W^{s,1}(S)\) with 
\[
s \in \bigl(0,\tfrac12\bigr]\qquad\text{and}\qquad\alpha \;>\; 1 - \Bigl(2 - \tfrac{1}{q}\Bigr) \, s,
\]
then \(\Haus^k(E) = 0\).
We show this exponents' condition to be sharp by constructing, for any \(\alpha < 1 - \bigl(2 - \tfrac{1}{q}\bigr) s\), a surface \(S \) in the same class as above and a set \(E \subset \tau(S,V)\) with \(\mathbb1_E \in W^{s,1}(S)\) and \(\Haus^k(E) > 0\).  
Our methods combine refined fractional Sobolev estimates on rectifiable sets, a Stokes‐type theorem for rough forms on finite‐perimeter sets, and a generalization of the Lusin's Theorem for gradients.


%58A30, 49Q15, 58A25, 53C17.
\par
}

\tableofcontents

\section{Introduction}\label{sec:intro}

\subsection{Main results}

A $C^1$-regular distribution of $k$-dimensional planes
\[
  V:\Omega\subset\R^{n}\;\longrightarrow\;\Gr(k,n)
\]
is \emph{integrable} if, through every point of~$\Omega$, one can construct a 
$C^2$-immersed, $k$-dimensional submanifold whose tangent space coincides with~$V$.
\medskip

Frobenius’ theorem asserts that this holds when $V$ is
\emph{involutive}, i.e. any two vector fields lying in $V$ have a Lie bracket that still lies in $V$.
In other words, a \emph{geometric} property (existence of integral manifolds)
is equivalent to a purely \emph{algebraic} constraint on the first derivatives of~$V$.
\bigskip

\noindent When the candidate surface is \emph{less} than $C^2$, the rigidity
of Frobenius’ theorem breaks down.  Indeed, the first-named author proved a
Lusin-type theorem in \cite{zbMATH00021945}, showing that there exist
$C^1$-regular surfaces tangent to a non-involutive distribution on a set of
positive measure. Z.~Balogh refined this in \cite{Balogh2003SizeGradient}, constructing
such surfaces to be of class 
\(\bigcap_{0<\alpha<1}C^{1,\alpha}\).  Note that in both Alberti's and Balogh's
examples the distribution $V$ must satisfy additional geometric invariance
properties, i.e. a vertical invariance of the distribution. A prototypical example is the horizontal distribution in Heisenberg groups; see \cite{Balogh2011SizeOT, Balogh2003SizeGradient}. In the case of the first Heisenberg group such distribution is spanned by the two vector fields in \eqref{distroH1}.

Without further hypotheses on the tangency set, this 
\(\bigcap_{0<\alpha<1}C^{1,\alpha}\) regularity is sharp: indeed if the surface were
$C^{1,1}$, then Frobenius together with Lusin’s theorem forces the tangency set
to have measure zero.

\medskip

\noindent These counterexamples demonstrate that the question of when a surface is tangent to a distribution is far from settled, and they have deep and far-reaching consequences in analysis on metric spaces. A prominent example arises in the theory of  Carnot-Carath\'eodory spaces. A Carnot-Carath\'eodory space is $\R^n$ equipped with a
distribution $V$ and the associated distance \(d_{cc}\), defined as the infimum
of lengths of paths tangent to $V$.  By Chow's theorem \cite[p.95, \S 0.4]{Gromov1996Carnot-CaratheodoryWithin}, in order to 
ensure that $d_{cc}$ is finite and defines a genuine distance  on $\R^n$ it is sufficient to have the H\"ormander condition, that asserts that the iterated Lie brackets of tangent vector fields to $V$ must span $\R^n$. 
Geometric Measure
Theory in these spaces is very active; see for instance \cite{MagnaniJEMS,MagnaniUnrect,ambled,Serapioni2001RectifiabilityGroup,FranchiSerapioni16,Merlo_2021, antonelli2023carnot,antonelli2020rectifiable2,Vittone20,HAL,antonelli2020rectifiableA,antonelli2020rectifiableB, MarstrandMattila20, kirchhserra,LeDonneLiRajala, ALD, ALDNG}. It is possible to prove that Lipschitz images of compact sets into Carnot-Carath\'eodory spaces are everywhere tangent
to $V$, and in fact inherit a $C^{1,1}$ structure.  For precise statements and
proofs, see \cite{alberti2025tangencysetsnoninvolutivedistributions}.  In
particular, in these spaces there are no rectifiable sets of dimension exceeding $\dim V$, and
the structure of the tangency set of a $C^{1,1}$ surface characterizes which
rectifiable sets occur in Carnot-Carath\'eodory spaces.
\bigskip

\noindent In the construction of the counterexamples given by the first-named author and Z. Balogh the set on which the $C^1$ surface is tangent to the given distribution is a fat fractal set. One might wonder if this irregularity must always manifest. In this direction, 
S.~Delladio proved in \cite{Del1,del2} that the tangency set of
any $C^1$-regular, $k$-dimensional surface to a non-involutive $C^1$
$k$-planes distribution cannot contain any finite-perimeter set.  Thus, even
for $C^1$ surfaces, the tangency set must have “many holes,” revealing a
trade-off between surface regularity and the regularity (of the boundary) of its tangency set.
\bigskip

\noindent From the above discussion, two questions arise naturally:
\begin{itemize}
  \item[(i)] Can one quantify a precise trade-off between the regularity of a
    surface and the irregularity of its tangency set to a non-involutive distribution?
  \item[(ii)] Does a Frobenius-type theorem hold for weaker objects, such as
    currents? If so, in what sense?
\end{itemize}
In this paper we resolve question~(i) and refine question~(ii), which
will be the subject of future research. The relationship between involutivity and the geometric structure of the boundary of a normal current has been in investigated in \cite{alberti2020geometric}.

\medskip

\noindent First of all, let us discuss the structure of tangency sets in the regimes in which Frobenius’ theorem
\emph{does} hold, namely for $C^2$ and $C^{1,1}$ submanifolds. In
\cite{alberti2025tangencysetsnoninvolutivedistributions}, the present authors
investigated the structure of tangency sets and their relation to different
degrees of non-involutivity of $V$ when the surfaces lie in these regularity
classes.
Recall that a $C^1$ $k$-plane distribution $V$ on $\R^n$ is called \emph{$h$-non-involutive},
with $2\le h\le k$, if every $C^1$ $h$-dimensional subdistribution
$W\subset V$ fails to be involutive. 
Under this definition, we obtained the following

\begin{theorem}[{\cite[Theorems 1.1.1, 1.1.2]{alberti2020geometric}}]\label{th:main1.intro}
  Let $2\le h\le k<n$ and let $V$ be an $h$-non-involutive $C^1$ $k$-plane
  distribution on $\R^n$.  If $S\subset\R^n$ is a $C^2$ submanifold of
  dimension $k$, then the set
  \[
    \tau(S,V):=\{\,q\in S : \Tan(S,q)=V(q)\},
  \]
  can be covered by countably many $(h-1)$-dimensional Lipschitz graphs.
If, on the other hand, $S$ is only $C^{1,1}$ of dimension $k$, then
  $\tau(S,V)$ is $h$-purely unrectifiable.
\end{theorem}

Although these results exclude tangencies with any $h$-dimensional
Lipschitz surface, they leave open how large \(\tau(S,V)\) can be in terms of
Hausdorff dimension in the $C^{1,1}$ case.  The first contribution of this paper is to settle this problem.

\begin{theorem}\label{teoremaC1,1}
  Let $2\le k<n$ and let $V$ be a $C^1$ $k$-plane distribution on $\R^n$.  Then
  for every $d<k$, there exists a $C^{1,1}$ submanifold
  $S\subset\R^n$ of dimension $k$ such that
  \[
    \dim_{\Haus}\bigl(\tau(S,V)\bigr)=d.
  \]
\end{theorem}
\medskip

\noindent This shows that lowering the surface regularity from $C^2$ to
$C^{1,1}$ allows the tangency set to attain any Hausdorff dimension below
$k$, while remaining $\Haus^k$-null by classical Frobenius’s theorem. This result extends \cite[Proposition 8.2(1)]{Balogh2011SizeOT} by removing the geometric constraints on the distributions $V$.

% Let us notice that the above results are quite precise, in the sense that in most cases in Geometric Measure Theory one would say that $C^2$ and $C^{1,1}$ surfaces are the same objects and one should expect the same kind of behaviour from them. However, since here we are investigating properties that are much finer than the quotient up to $\Haus^k$-null sets, we therefore obtain different results. Indeed, the fact that we pass from a continuous second derivative to an only $L^\infty$ one shows that the tangency set lose a lot of structure, even resulting in the fact that they can have an arbitrarily big dimension.

\medskip

Now that we have completed a fine analysis for tangency sets in the regimes of smoothness where Frobenius' theorem holds, we pass to dig deeper in the cases in which the regularity is strictly below $C^{1,1}$, where essentially we do not assume any second-order regularity. As anticipated above, we should expect a trade-off phenomenon to emerge between the regularity of the surface and that of tangency sets in order for Frobenius-type theorems to hold.
To state our results, we need a notion of regular subsets of rectifiable sets. This class is introduced in \S\ref{parag:def:sobolev:surface} and generalizes the standard definition of Sobolev–Slobodeckij functions on $\R^n$ to general rectifiable sets. It is worth noting that, as with the standard definition of finite‐perimeter sets in arbitrary open sets, our definitions do not detect any irregularity of the boundary of the rectifiable set itself. With this definition we can obtain the following striking Frobenius-type result.

\begin{theorem}
    Let $V$ be a $k$-dimensional distribution in $\mathbb{R}^n$ of class $C^1$. Suppose $S$ is a $k$-dimensional rectifiable set. Suppose 
$$E \subseteq \tau(S, V)\cap N(V),$$ 
is a Borel set, where $\tau(S, V)$ denotes the tangency set defined in \S\ref{def:tan} and $N(V)$ denotes the non-involutivity set of the distribution $V$, see \S\ref{def:inv}. If the characteristic function $\mathbb{1}_E$ belongs to $W^{s, 1}(S)$ with $s>1/2$, where  the space $W^{s,1}(S)$ was introduced in  \S\ref{parag:def:sobolev:surface}, then $\mathcal{H}^k(E) = 0$.
\end{theorem}

This theorem shows that Delladio's result can be pushed further. Not only tangency sets cannot be of finite perimeter inside the surface, but they cannot even have finite $1/2$-perimeter. Below the threshold of $1/2$-regularity, as we should see below, the situation is more complicated.  In what follows we say that a $k$-dimensional rectifiable set is of class $Y^{1+\alpha,q}$ if it can be covered with countably many $C^1$ graphs whose gradients are of class $W^{\alpha,q}$. See Definition \ref{Yreg} for a precise definition. With this definition we are able to obtain the following extension of Frobenius' theorem.

\begin{theorem}\label{main:intro:frobenius}
Let $V$ be a $k$-dimensional distribution in $\mathbb{R}^n$ of class $C^1$, and let $ q \in [1, \infty]$, $\alpha \in (0,1)$, $s\in (0,1/2]$ be such that 
\[
\alpha > 1 - \left(2 - \frac{1}{q} \right) s.
\]
Suppose $S$ is a $k$-dimensional $Y^{1+\alpha, q}$-rectifiable set, see \S\ref{def:rectifiable} and \S\ref{Yreg} for a formal definitions, and let 
$$E \subseteq \tau(S, V)\cap N(V),$$ 
be a Borel set, where $\tau(S, V)$ denotes the tangency set defined in \S\ref{def:tan} and $N(V)$ denotes the non-involutivity set of the distribution $V$, see \S\ref{def:inv}. If the characteristic function $\mathbb{1}_E$ belongs to $W^{s, 1}(S)$, where  the space $W^{s,1}(S)$ was introduced in  \S\ref{parag:def:sobolev:surface}, then $\mathcal{H}^k(E) = 0$.
\end{theorem}

Theorem \ref{main:intro:frobenius} is in fact sharp, as the following result shows.

\begin{theorem}\label{lusin:intro}
    Suppose $V$ is a $k$-dimensional distribution of class $C^1$ and let $ q \in [1, \infty]$, $s, \alpha \in (0,1)$ be such that $s< 1/2$ and
\[
\alpha < 1 - \left(2 - \frac{1}{q} \right) s.
\]
Then, there exists an embedded $k$-dimensional submanifold $S$ of class $Y^{1+\alpha,q}$ and a Borel set $E$ with $\mathbb{1}_E\in W^{s,1}(S)$ such that 
$$E\subseteq \tau(S,V) \qquad\text{and}\qquad \Haus^k(E)>0.$$
\end{theorem}

For the sake of discussion of the content of Theorems \ref{main:intro:frobenius} and \ref{lusin:intro}, let us say that $q=\infty$ and thus $S$ is of class $C^{1,\alpha}$ and that $\mathbb{1}_{\tau(S,V)}\in W^{s,1}$. Theorem  \ref{main:intro:frobenius} tells us, and this is the reason for which it is an extension of Frobenius' theorem, that is, if $\alpha>1-2s$, then the tangency set $\tau(S,V)$ must be $\Haus^k$-null. On the other hand, if $\tau(S,V)$ has positive measure then $\tau(S,V)$ cannot be a set of finite $s$-fractional perimeter
% , see \cite{Citare}, 
for every $s>(1-\alpha)/2$. In other words, the more regular the surface is, the worse the boundary of the tangency set to a non-involutive distribution gets.

Finally, the last result connected to Frobenius' theorem we provide is the following generalization of \cite[Proposition 8.2(2)]{Balogh2011SizeOT} that reads

\begin{theorem}\label{lusin:extremal}
    Suppose $V$ is a $k$-dimensional distribution in $\R^n$ of class $C^1$. Then, there exists a submanifold of class $\bigcap_{0<\alpha<1}C^{1,\alpha}$ of $\R^n$ such that 
    $$\Haus^k(\tau(S,V))>0.$$
\end{theorem}

Notice that, in light of the above discussion, Theorem \ref{lusin:extremal} is sharp in the following sense: if one takes Theorem \ref{main:intro:frobenius} and lets $\alpha$ be arbitrarily big, we see that we cannot expect the tangency set to contain any set of finite fractional perimeter, even if we require the surface to be of regularity $\bigcap_{0<\alpha<1} Y^{1+\alpha,1}$. Rephrasing, in this regime we must expect the tangency set to be the typical Borel set, without any enhanced regularity of its boundary. 

In the following picture we summarize the values of $\alpha$ and $s$ for which we have counterexamples or Frobenius-type theorem, for a fixed $q\in [1,\infty]$. Notice that, with our arguments, we cannot decide what happens at the boundary between the regions of counterexamples (orange region) or the Frobenius-type theorems (blue region).

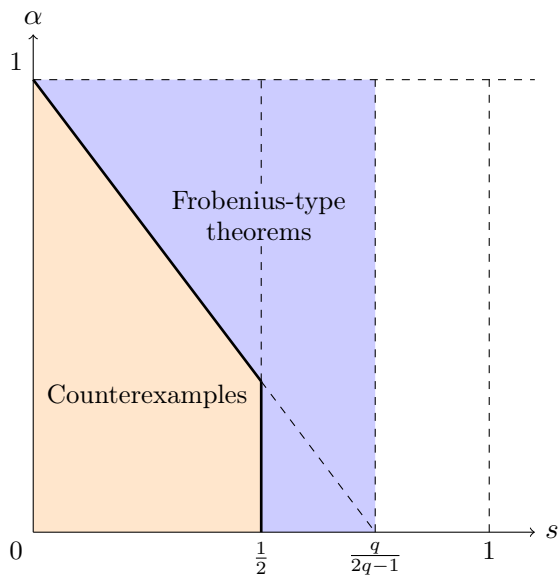
\begin{figure}[ht]
  \centering
  \begin{tikzpicture}[scale=6]
    % === parameters ===
    \def\q{3/2}                           % <-- change q here
    \pgfmathsetmacro\xone{1/2}          % x = 1/2
    \pgfmathsetmacro\xtwo{\q/(2*\q-1)}  % x = q/(2q-1)
    % intersection of line (0,1)–(xtwo,0) with x = xone
    \pgfmathsetmacro\yone{1 - (\xone/\xtwo)}

  \node[anchor=north east,font=\small] at (0,0) {$0$};

    % === shaded regions ===
    \fill[orange!20]
      (0,0) -- (0,1) -- (\xone,\yone) -- (\xone,0) -- cycle;
    \fill[blue!20]
      (0,1) -- (\xtwo,1) -- (\xtwo,0)-- (\xone,0) -- (\xone,\yone) -- cycle;

    % === axes ===
    \draw[->] (0,0) -- (1.1,0) node[right] {$s$};
    \draw[->] (0,0) -- (0,1.1) node[above] {$\alpha$};

    % === ticks & labels ===
    \node[above left,font=\small]  at (0,1)    {$1$};
  \node[below,font=\small]      at (\xone,0) {$\tfrac12$};
  \node[below,font=\small]      at (\xtwo,0) {$\tfrac{q}{2q-1}$};
  \node[below,font=\small]      at (1,0)     {$1$};
    % \draw (0,0) -- ++(0,-2pt) node[below] {$0$};

    % === dashed guide lines ===
    \draw[dashed] (0,1)    -- (1.1,1);
    \draw[dashed] (\xtwo,0) -- (\xtwo,1);
    \draw[dashed] (1,0)    -- (1,1);

    % === slanted boundary: split at (xone,yone) ===
    \draw[line width=1pt]         (0,1) -- (\xone,\yone);
    \draw[dashed]  (\xone,\yone) -- (\xtwo,0);

    % === vertical at x=1/2: solid--dashed split ===
    \draw[line width=1pt]         (\xone,0) -- (\xone,\yone);
    \draw[dashed]  (\xone,\yone) --(\xone,0.65);
    \draw[dashed]  (\xone,0.77) --(\xone,1);

    % === region labels ===
    \node[align=center,font=\small] at (0.25,0.3) {Counterexamples};
    \node[align=center,font=\small] at 
      ({0.27 + (\xtwo - 0.3)/2},0.7) {Frobenius-type\\theorems};
  \end{tikzpicture}
\end{figure}

Let us remark that obtaining such sharp, gapless results with our techniques was really surprising since the construction of the counterexamples and the proof of the Frobenius-type theorems are based on completely different ideas.

Nevertheless, this family of results completely settles our first question, and shows what we should expect for the second question. Before proceeding with the discussion of the ideas of the proof, let us explore future directions and open problems. 

\smallskip

\subsection{Future directions}

One of the natural generalizations of surfaces are currents, therefore a natural question to ask would be the following. 

\begin{itemize}
    \item[(Q)] Suppose that $T$ is a $k$-rectifiable current tangent to a $C^1$ non-involutive distribution. If $T$ has a boundary of fractional regularity $s\in(0,1]$, is it true that, if $s>1/2$, then $T=0$?
\end{itemize}

Notice that the above question is still quite not well defined, since before answering (Q) one needs to establish what fractional regularity for a boundary of a current means. A first result in this direction can be obtained when $V$ is smooth and enjoys the H\"ormander condition, that will appear in \cite{frobenius4}.

\begin{theorem}[{\cite{frobenius4}}]\label{annuncio}
      Suppose that $T$ is a $k$-current with finite mass in $\R^n$ and let $\tau=\frac{dT}{d\lVert T\rVert}$ be its polar $k$-vector, meaning that $T=\tau\mu$ with $\mu$ Radon measure. Let $V$ be a smooth $k$-dimensional distribution satisfying the H\"ormander condition, and assume that 
    $\mathrm{span}(\tau)=V$ for $\lVert T\rVert$-almost every $x\in \R^n$,
    where the span of a $k$-vector was introduced in \cite[\S 2.3] {alberti2020geometric}. 
    Suppose further that there exists an $\alpha\in (0,1]$ such that 
    $$\mathbb{M}[T-(\Phi_{h}^X)^\#T]\lesssim \lvert h\rvert^\alpha,$$
    where $X$ is a $C^1$ vector field tangent to $V$, $\phi_h^X$ is the flow of said vector field at time $h$ and $(\Phi_{h}^X)^\#T$ is the pullback of $T$ under the map $\Phi_{h}^X$. Then $\mu\ll \Leb^n$.
\end{theorem}

Notice that the condition $\mathbb{M}[T-(\Phi_{h}^X)^\#T]\lesssim \lvert h\rvert^\alpha$ encodes a fractional-type regularity for the boundary. For instance, if $k=n$ and $\alpha=1$ this characterizes $n$-dimensional normal currents, i.e., $BV$ functions. Essentially the fractional regularity of the boundary and the fact that the iterate commutators of $V$ span $\R^n$ imply that $\mu$ is diffuse. The connection with deep PDE results like \cite{DPR} and \cite{ReversePansu} is clear when the current is normal, however when the boundary becomes only a distribution those techniques break down.

In order to tackle questions like (Q), following the approach of Theorem \ref{annuncio} would mean to show that, if the boundary is too regular, one can prove that the total variation of $T$ must be absolutely continuous to a Hausdorff measure of dimension strictly bigger than $k$. Indeed, this will require further, non-trivial work.

\subsection*{Ideas of the proofs}

To illustrate the core mechanism, we focus on the simplest non‐involutive case: in $\R^3$ let 
\begin{equation}
  V(x) = \mathrm{span}\{X(x),Y(x)\},
  \quad\text{where}\quad
  X(x)=e_1 - 2x_2\,e_3,\quad
  Y(x)=e_2 + 2x_1\,e_3.
    \label{distroH1}
\end{equation}

Define the linear map $M(x):\R^2\to\R^3$ by 
\[
  M(x)e_1 = X(x), 
  \quad
  M(x)e_2 = Y(x).
\]
A direct calculation shows that any Lipschitz function 
\[
  f:K\Subset\R^2\;\longrightarrow\;\R
  \quad\text{whose graph } \Gamma=\{(x,f(x)):x\in K\}
\]
is everywhere tangent to $V$ on $K$ if and only if 
\[
  Df(x) = M(x)
  \qquad\text{for almost every }x\in K.
\]
But if this identity actually held on an open neighborhood, then $f$ would inherit two continuous partial derivatives and thus be $C^2$, contradicting Schwarz’s theorem on equality of mixed partial derivatives.
The key takeaway is that whenever $Df=M$ on a \emph{large} set and $f$ enjoys too much regularity, one forces a forbidden equality of mixed derivatives.  Hence surface regularity and size of the tangency set must balance each other.

The proof of our fractional Frobenius theorem, Theorem \ref{main:intro:frobenius}, follows a similar philosophy.  Building on Delladio’s insight \cite{Del1,del2}, one shows that if a Borel set $E\subset\R^n$ has indicator $\mathbb 1_E\in W^{s,1}$, then $E$ satisfies a \emph{super‐density} property that is, for almost every $x\in E$,
\[
  \lim_{r\to0}
  \frac{\Leb^n\bigl(B(x,r)\cap E\bigr)}{r^{n+s^*}}
  = 0,
  \quad
  s^*=\frac{n}{n-s}.
\]
This follows from Dorronsoro’s differentiability theorem for Besov functions, see \cite{dorronsoro}.
Using super‐density in conjunction with a Stokes–type theorem for rough forms on finite‐perimeter sets, see Proposition~\ref{divergencethm}, and a suitable Poincaré/Morrey's estimate, depending on the Sobolev/Hölder regime, one proves the following

\begin{theorem}[(Locality of the divergence)]\label{localitydiv}
  Let $g\in W^{\alpha,q}(\R^2;\R^2)$ be continuous and vanish on a Borel set $E$ with $\mathbb 1_E\in W^{s,1}(\R^2)$.  If
  \[
    \alpha \;>\; 1 - \Bigl(2-\tfrac1q\Bigr)\,s,
  \]
  then $\mathrm{d} g\equiv0$ in the distributional sense.
\end{theorem}
Restricting this divergence‐nullity to the two non‐commuting vector fields in our model reduces Theorem~\ref{localitydiv} to Theorem~\ref{main:intro:frobenius} in arbitrary codimension, see Proposition~\ref{prop:delicata2}.

On the constructive side, we extend Alberti’s Lusin-type theorem for gradients \cite{zbMATH00021945} to a full fractional setting.

 \begin{theorem}\label{lusinregSobolev:intro}
         Let $\eta,\varepsilon>0$, $q\in[1,\infty]$ and $\alpha\in [0,1)$ and $0\leq s<q/(2q-1)$ such that 
         $$\alpha<1-\Big(2-\frac{1}{q}\Big)s.$$
         Let $\Omega$ be an open bounded set in $\R^k$ and suppose that $F:\Omega\times \R^{n-k}\to \R^{k\times {n-k}}$ is a locally Lipschitz map. 
         Then, there  are a compact set $\mathfrak{C}\subseteq \Omega$ and a function $u:\Omega\to \R^{n-k}$ such that
         \begin{itemize}
         \item[(i)]$\mathrm{supp}(u)\subseteq \Omega$, $\lVert u\rVert_{\infty}\leq \eta$ and $Du(x)=F(x,u(x))$ for $\Leb^k$-almost every $x\in \mathfrak{C}$;
             \item[(ii)] 
$\mathscr{L}^k(\Omega\setminus \mathfrak{C})\leq \varepsilon\mathscr{L}^k(\Omega)$ and $\mathbb{1}_{\mathfrak{C}}\in W^{s,1}(\Omega)$;
\item[(iii)] $u\in L^\infty(\Omega)\cap W^{1,q}(\Omega)$ and $Du\in W^{\alpha,q}(\Omega)$.
 \end{itemize}
In addition, if $0\leq s<1/2$ then $u$ is also of class $C^1_c(\Omega)$ and the identity 
$$Du(x)=F(x,u(x)) \qquad\text{holds everywhere on $\mathfrak{C}$}.$$
Finally, if $s=0$, then $u\in \bigcap_{0<\alpha<1}C^{1,\alpha}(\Omega)$. 
    \end{theorem}

This Lusin-for-gradient-type result allows us to construct $C^1$ surfaces tangent to general distributions of $k$-planes in $\R^n$ with varying degree of regularity for its tangent field.

{\small
\begin{parag*}[Acknowledgements]
Part of this research was carried out during visits
of the second and third authors at the Mathematics Department in Pisa.
These were supported by the University of Pisa through 
the 2015 PRA Grant ``Variational methods for geometric problems'', by the 2018 INdAM-GNAMPA 
project ``Geometric Measure Theoretical approaches to Optimal Networks'' and by PRIN project MUR 2022PJ9EFL ``Geometric Measure Theory: Structure of Singular Measures, Regularity Theory and Applications in the Calculus of Variations''.
   
The research of G.A. has been partially supported
by the Italian Ministry of Education, University and Research (MIUR) 
through the PRIN project 2010A2TFX2
%``Calcolo delle Variazioni''. 
The research of A. Ma has received funding from the European Union's Horizon 2020 research and innovation programme under the grant agreement No. 752018 (CuMiN), STARS@unipd research grant ``QuASAR -- Questions About Structure And Regularity of currents'' (MASS STARS MUR22 01), National Science Foundation under Grant No. DMS-1926686 and INdAM project ``VAC\&GMT''.  The research of A. Me has has received funding from the European Union's Horizon 2020 research and innovation programme under the Marie Sk\l odowska-Curie grant agreement no 101065346.
%``Currents and Minimizing Networks''.

\smallskip

The authors wish to thank Tuomas Orponen for precious comments. 
\end{parag*}
}

\section{Notation and preliminary results}
\label{sec:not}
Here is a list of frequently used notations:
\begin{itemize}
[leftmargin=50pt, itemsep=4 pt plus 1pt, labelsep=10pt]

\item[{$\lvert \cdot\rvert$}] Euclidean norm;

% \item[{$d_V(\cdot,\cdot)$}] Carnot-Carath\'eodory metric relative to the distribution of planes $V$ (\S\ref{def:tan});

\item[$B(x,r)$] Euclidean open ball centerd at $x$ with radius $r$.

\item[${\text{diam}}$] diameter of a set with respect to the distance $d$. If $d=\lvert\cdot\rvert$ is the Euclidean metric we drop the subscript.
 
%\item[{$\rho\mu$}]
%measure associated to a measure $\mu$ on $X$ and a (Borel) density %$\rho$, 
%that is, $[\rho\mu](E) := \int_E \rho \, d\mu$ for every Borel set $E$ %in $X$;

%\item[{$f_\#\mu$}]
%push-forward of a measure $\mu$ on $X$ according 
%to a Borel map $f:X\to Y$, that is, 
%$[f_\#\mu](E) := \mu(f^{-1}(E))$ for every Borel set $E$ in $Y$;

%\item[{$|\mu|$}]
%variation measure associated to a signed measure $\mu$;

%\item[{$\mu\ll\lambda$}]
%the measure $\mu$ is absolutely continuous \wrt the measure
%$\lambda$;

%\item[{$\mu_a$, $\mu_s$}]
%absolutely continuous and singular part of the
%Lebesgue decomposition of $\mu$ \wrt the 
%some measure $\lambda$.

\item[{$\Leb^n$}]
Lebesgue measure on $\R^n$;

\item[{$\Haus^\alpha$}]
$\alpha$-dimensional Hausdorff measure on $\R^n$;

\item[{$\Tan(S,x)$}] tangent plane to the surface $S$ at the point $x$;

%\item[{$\I^d_t$}]
%$d$-dimensional integral geometric measures (\S\ref{def:intgeo});

%\item[$I(n,k)$]
%set of all multi-indices %$\bfi:=(i_1,\dots,i_k)$ 
%with $1\le i_1<\dots<i_k\le n$;

\item[$\M(n,k)$]
set of linear maps from $\R^n$ to $\R^k$;

\item[$\Gr(k,n)$] $k$-dimensional Grassmanian, 

%\item[$\esterno_k(V)$]
%space of $k$-vectors in a linear space $V$;  
%the canonical basis of $\esterno_k(\R^n)$ is formed by the 
%simple $k$-vectors $e_\bfi:=e_{i_1}\wedge\dots\wedge e_{i_k}$ 
%with $\bfi\in I(n,k)$, where $\{e_1,\dots,e_n\}$
%is the canonical basis of $\R^n$;
%$\esterno_k(\R^n)$ is endowed with the Euclidean
%norm $|\cdot|$ associated to this basis, that is, 
%$\smash{ |v|^2:=\sum_\bfi v_\bfi^2 }$
%(note however that none of the results in this paper
%depend on the specific choice of the norm);

%\item[$\de x$]
%$:= \de x_1\wedge\dots\wedge \de x_n$;

\item[$\langle v, v'\rangle$]
scalar product of the vectors $v,v'\in\R^n$;

%\item[$\scalar{v}{\alpha}$]
%duality pairing of the $k$-vector $v$ and the $k$-covector $\alpha$;
%}

%\item[$\wedge$]
%exterior product (of $k$-vectors or $h$-covectors);

%\item[$\antitrace \, , \, \trace$]
%interior products of a $k$-vector and a $h$-covector
%(\S\ref{def:int_prod}); 

%\item[$\hodge$]
%operator on multi-vectors and covectors defined in \S\ref{def:hodge};

%\item[$\Span(v)$]
%span of a $k$-vector $v$ (\S\ref{def:span});

%\item[$\de$]
%exterior derivative of a $k$-form (\S\ref{def:diffop}); 
%differential of a map;

%\item[$\dive$]
%divergence of a $k$-vector field (\S\ref{def:diffop});

\item[{$[v,v']$}]
Lie bracket of vector fields $v$ and $v'$ (\S\ref{def:lie});

%\item[$\scalar{T}{\omega}$]
%duality pairing of the $k$-current $T$ and the $k$-form $\omega$;
%
%\item[$\bd T$, $\Mass(T)$]
%boundary and mass of the current $T$ (\S\ref{def:curr});
%
%\item[$T\trace\omega$]
%interior product of the $k$-current $T$ and the 
%$h$-form $\omega$ (\S\ref{s-intprod});
%
%\item[{$\IC{E,\tau,m}$}]
%current associated to a rectifiable set $E$, 
%an orientation $\tau$, and a multiplicity
%$m$ (\S\ref{s-intcurr});

%\item[$W(\mu,\cdot)$]
%decomposability bundle of a measure $\mu$ (\S\ref{def:decbun});

\item[{$N(V)$}]
non-involutivity set of a distribution of $k$-planes $V$ (\S\ref{def:inv});

%\item[{$\Vhat$}]
%distribution associated to a distribution of $k$-planes $V$ as in \S\ref{def:vhat}.
\end{itemize}

\medskip

\subsection{Fractional Sobolev spaces: regularity and representation}

\begin{parag}[H\"older spaces]
\label{def:Hol}
Let $\Omega$ be an open subset of $\mathbb{R}^n.$ 
Assume that $\alpha\in(0,1).$ 
% Rewrote to avoid abbreviations and dimension-changing parentheses
For any function $f:\Omega\to\mathbb{R},$ we define its $\alpha$-H\"older seminorm as
\[
[f]_\alpha := \sup_{x\neq y \in \mathrm{cl}(\Omega)} 
\frac{\lvert f(x)-f(y)\rvert}{\lvert x-y\rvert^\alpha},
\]
and we set $\lVert f\rVert_\alpha := [f]_\alpha + \lVert f\rVert_\infty$.
Define
\[
C^{\alpha}(\Omega) := 
\bigl\{ f\in{C}(\Omega) : \lVert f\rVert_\alpha < \infty\bigr\}.
\]
It is straightforward to verify that the normed space $\bigl(C^{\alpha}(\Omega),\lVert \cdot\rVert_\alpha\bigr)$ is complete and commonly referred to as the space of $\alpha$-H\"older continuous functions.
% We also denote by $C^{\alpha+}(\Omega)$ the closure in $C^{\alpha}(\Omega)$ (with respect to the norm $\lVert\cdot\rVert_\alpha$) of the smooth, compactly supported functions $C^\infty_c(\Omega).$
\end{parag}

\begin{parag}[Sobolev-Slobodeckij spaces]
\label{Walphauno}
Let $\Omega$ be an open subset of $\mathbb{R}^n.$ 
Assume $s\in (0,1]$ and $p\in [1,\infty].$
For any measurable map $v:\Omega\to\mathbb{R}^{m},$ we denote by $[v]_{W^{s,p}(\Omega)}$ the seminorm
\[
[v]_{W^{s,p}(\Omega)}^p 
:= \int_{\Omega}\int_{\Omega}
\frac{\lvert v(x)-v(y)\rvert^p}{\lvert x-y\rvert^{sp + n}}\,dx\,dy.
\]
As usual, we denote by $W^{s,p}(\Omega)$ the following complete metric space:
\[
W^{s,p}(\Omega) := \bigl\{v\in L^p(\Omega,\R^m) : [u]_{W^{s,p}(\Omega,\R^m)} < \infty\bigr\}.
\]
Notice that if $p=\infty$ then $W^{s,p}(\Omega,\R^m)=C^{s}(\Omega,\R^m)$.
\end{parag}

The following theorem is a beautiful consequence of approximate differentiability of Besov functions that was obtained by J. Dorronsoro.

\begin{theorem}[{\cite[Theorem 2]{dorronsoro}}] \label{dorronsorox}
Let $s\in(0,1]$, $p\in[1,\infty)$, with $s\leq n/p$ and let $b<s$. Then, for every $v\in W^{s,p}(\R^n)$, we have 
    $$\lim_{t\to 0} t^{-b} \Big(\fint_{\lvert y\rvert\leq t}\lvert v(x+y)-v(x)\rvert^{p^*}d\Leb^n(y)\Big)^{\frac{1}{p^*}}=0\qquad\text{for $\Leb^n$-almost every }x\in \R^n.$$
\end{theorem}

\begin{remark}
    In case $sp=n$, then $p^*=\infty$ and the above formula has to be understood as
    $$\lim_{t\to 0}t^{-b}\sup_{\lvert y\rvert\leq t}\lVert v(\cdot+y)-v(\cdot)\rVert_\infty=0.$$
\end{remark}

Specializing Theorem \ref{dorronsorox} to the case in which the function $v$ is the indicator function of some Borel set $E$, we get the following \emph{super-density} result for $E$.

\begin{proposition}\label{dorronsorosuper-density}
Let $s\in(0,1)$, with $s\leq n$ and let $B$ be any open ball in $\R^n$. Assume further that $E$ is a $\Leb^n$-measurable set such that $\mathbb{1}_E\in W^{s,1}(B)$. Then, for every $0\leq b<s$ we have
$$\lim_{\rho\to 0}\frac{\Leb^n(B(x,\rho)\setminus E)}{\rho^{n+b1^*}}=0,\qquad \text{for $\Leb^n$-almost every $x\in E$,}$$
where as usual $1^*=n/(n-s)$ denotes the Sobolev exponent associated to $n,s$ and the expo $1$.
\end{proposition}

\begin{remark}
    The above result shows that an increase in the regularity of a set implies high density locally, or more specifically an improved Lebesgue's differentiation theorem. 
\end{remark}

\begin{proof}
By \cite[Theorem 5.4]{HGfractionalsobolev}, we can extend $\mathbb{1}_E$ to a function $u\in W^{s,1}(\R^n)$ and by Theorem \ref{dorronsorox}, we know that
for $\Leb^n$-almost every $x\in \R^n$ we have 
\begin{equation}
    \label{dorronsoroimplic}
    \lim_{\rho\to 0}\rho^{-(n+b\,1^*)}\int_{B(0,\rho)}\lvert u(x+h)-u(x)\rvert^{1^*}\, d\Leb^n(h)=0.
\end{equation}
Notice however, that whenever $x\in B$, we have 
\begin{equation}
\begin{split}
      \lim_{\rho\to 0}\rho^{-(n+b\,1^*)}\int_{B(0,\rho)}\lvert u(x+h)-&u(x)\rvert^{1^*}\, d\Leb^n(h)\\
    =
    \lim_{\rho\to 0}\rho^{-(n+b\,1^*)}&\int_{B(0,\rho)}\lvert \mathbb{1}_E(x+h)-\mathbb{1}_E(x)\rvert^{1^*}\, d\Leb^n(h).
    \nonumber
\end{split}
\end{equation}
In particular, this in turn implies that for $\Leb^n$-almost all $x\in E$ one has
\begin{equation}
    \begin{split}
        & \lim_{\rho\to 0}\frac{\Leb^n(B(x,\rho)\setminus E)}{\rho^{n+b\, 1^*}}\\
&\qquad\qquad\qquad=\lim_{\rho\to 0}\rho^{-(n+b\,1^*)}\int_{B(0,\rho)}\lvert \mathbb{1}_E(x+h)-\mathbb{1}_{B(x,\rho)}(x+h)\rvert\, d\Leb^n(h)\\
         &\qquad\qquad\qquad= \lim_{\rho\to 0}\rho^{-(n+b\,1^*)}\int_{B(0,\rho)}\lvert \mathbb{1}_E(x+h)-1\rvert^{1^*}\, d\Leb^n(h)\overset{\eqref{dorronsoroimplic}}{=}0.
         \nonumber
    \end{split}
\end{equation}
This concludes the proof.
\end{proof}

We now state a simple Poincare-type inequality for fractional Sobolev functions. 

\begin{proposition}\label{poincarefractional}
    Let $\alpha\in (0,1)$, $q\in [1,\infty]$ with $\alpha q>n$. There exists a constant $c>0$ depending only on $\alpha$, $q$ and $n$ such that for every $R>0$ we have 
    $$\lVert u\rVert_{L^1(B(0,R))}\leq cR^{n(1-1/q)+\alpha}[u]_{W^{\alpha,q}(B(0,R))},$$
    whenever $u=0$ on $\partial B(0,R)$.
\end{proposition}

\begin{remark}
    Let us observe that although writing $u=0$ on $\partial B(0,R)$ is imprecise since functions in $W^{\alpha,q}(B(0,1))$ are defined only $\Leb^n$-almost everywhere. However, Morrey's inequality, see \cite[Theorem 8.2]{HGfractionalsobolev} guarantees that if $\alpha  q>n$, then $u$ must have a $C^{\alpha-n/q}(B(0,1))$ representative and hence the trace at $\partial B(0,1)$ is just determined by a restriction of any $C^{\alpha-n/q}$-extension of (the continuous representative of) $u$ to $\partial B(0,1)$. 
\end{remark}

\begin{proof}
As a first step, let us assume $R=1$.    By contradiction suppose that this is not the case and that there exists a sequence of continuous functions $u_i\in W^{\alpha,q}(B(0,1))$ such that $u_i=0$ on $\partial B(0,1)$, $\lVert u_i\rVert_{L^1(B(0,1))}=1$ and 
    $$[u_i]_{W^{\alpha,q}(B(0,1))}\leq i^{-1}\lVert u_i\rVert_{L^1(B(0,1))}=i^{-1}.$$
    Thanks to \cite[Theorem 7.1]{HGfractionalsobolev} we know that up to subsequences there exists a function $u\in L^1(B(0,1))$ such that $u_i$ converges to $u$ in $L^1(B(0,1))$. Let us further notice that thanks to \cite[Theorem 8.2]{HGfractionalsobolev} the sequence $u_i$ is uniformly equicontinuous and equibounded, thanks to our assumption that $u_i=0$ on $\partial B(0,1)$. Therefore the function $u$ is also $C^{\alpha-n/q}$ and $u=0$ on $\partial B(0,1)$. 

 Therefore, Fatou's lemma implies that 
    $$[u]_{W^{\alpha,q}(B(0,1))}\leq \liminf_{i\to \infty}[u_i]_{W^{\alpha,q}(B(0,1))}=0.$$
    This shows that $u$ is constant in $B(0,1)$. However, the continuity of $u$ and the fact that $u=0$ on $\partial B(0,1)$, implies that $u=0$ in $B(0,1)$. This is in contradiction with the fact that $\lVert u\rVert_{L^1(B(0,1))}=1$. 

Let us conclude the proof by rescaling. If $u\in W^{\alpha,q}(B(0,R))$ then
\begin{equation}
\begin{split}
     &\qquad\qquad\qquad\qquad\int_{B(0,R)}\lvert u(z)\rvert dz= R^n\int_{B(0,1)}\lvert u(Rw)\rvert dw\\
    \leq& cR^n\Big(\int_{B(0,1)} \int_{B(0,1)} \frac{\lvert u(Rw)-u(Rz)\rvert^q}{\lvert w-z\rvert^{n+\alpha q}}dz\,dw\Big)^\frac{1}{q}= cR^{n(1-1/q)+\alpha}[u]_{W^{\alpha,q}(B(0,R))}.
    \nonumber
\end{split}
\end{equation}
This concludes the proof.
\end{proof}

\begin{definition}
    For every $1\leq k\leq n$, we denote by $\mathrm{Gr}(k,n)$ the Grassmannian of $k$-dimensional planes in $\R^n$. In addition, we denote by $\gamma_{k,n}$ the Radon measure on $\mathrm{Gr}(k,n)$ that is invariant under the action of the orthogonal group $O(n)$. For the existence of such a measure we refer to \cite[\S 3.5]{Mattila1995GeometrySpaces}.
\end{definition}

The following is a technical proposition that will be employed in the proof of Proposition \ref{restrictionW1,p}.

\begin{proposition}\label{prop:polarhighdim}
Let $1\leq k\leq n$. Then, there exists a constant $c_{k,n}$ such that for every positive Borel function $f:\R^n\to \R$, we have
$$\int f(z)dz=c_{k,n}\int_{V\in\mathrm{Gr}(k,n)}\int f(w)\lvert w\rvert^{n-k}\, d\Haus^{k}\llcorner V(w)\, d\gamma_{k,n}(V).$$ 
\end{proposition}

\begin{proof}
As a first step, we prove that the measure $\nu$, defined for every Borel set $A$ by
$$\nu(A)=\int_{V\in\mathrm{Gr}(k,n)}\Bigl(\int \mathbb{1}_A(h)\lvert h\rvert^{n-k}\, d\Haus^{k}\llcorner V(h)\Bigr)d\gamma_{k,n}(V),$$ % Added \Bigr for clarity
is absolutely continuous with respect to the Lebesgue measure $\Leb^n$. 
Thanks to \cite[Lemma 3.11]{Mattila1995GeometrySpaces}, we know that 
$$\gamma_{k,n}(\{V\in \Gr(k,n):\mathrm{dist}(x,V)\leq r\})\leq 2^{n}\omega_n^{-1}r^{n-k}\lvert x\rvert^{k-n},$$
for every $x\in\R^n\setminus \{0\}$ and every $r\leq \lvert x\rvert$. Thus, it follows that 
\begin{equation}
    \nu(B(x,r))\leq  2^{n}\omega_n^{-1}r^{n-k}\lvert x\rvert^{-(n-k)}\cdot \omega_k2^{n-k}r^k\lvert x\rvert^{n-k}=2^{2n-k}\frac{\omega_k}{\omega_n}r^n,
\end{equation}
which immediately implies that $\nu\ll\Leb^n$. 

In addition, the measure $\nu$ is invariant under the action of the orthogonal group $O(n)$, since $\gamma_{k,n}$ is invariant, see, e.g., \cite[Chapter 3]{Mattila1995GeometrySpaces}.
This invariance implies that there exists a function $g:[0,\infty)\to[0,\infty]$ such that $\nu=g(\lvert \cdot \rvert)\Leb^n$. Furthermore, $\nu$ is an $n$-dimensional cone, i.e., 
$$\nu=\lambda^{-n}T_{0,\lambda}\nu$$ 
for every $\lambda>0$. Indeed, for every Borel set $A$, we have
\begin{equation}
\begin{split}
        \lambda^{-n}\nu(\lambda A)
        &=\lambda^{-n}\int_{V\in\mathrm{Gr}(k,n)}\Bigl(\int \mathbb{1}_{A}(\lambda^{-1}h)\lvert h\rvert^{n-k}\, d\Haus^{k}\llcorner V(h)\Bigr)d\gamma_{k,n}(V)\\
         &=\lambda^{-n}\int_{V\in\mathrm{Gr}(k,n)}\Bigl(\int \mathbb{1}_{A}(\lambda^{-1}h)\lambda^{n-k}\lvert \lambda^{-1}h\rvert^{n-k}\, d\Haus^{k}\llcorner V(h)\Bigr)d\gamma_{k,n}(V)\\
         &=\int_{V\in\mathrm{Gr}(k,n)}\Bigl(\int \mathbb{1}_{A}(w)\lvert w\rvert^{n-k}\, d\frac{T_{0,\lambda}\Haus^{k}\llcorner V}{\lambda^k}(w)\Bigr)d\gamma_{k,n}(V)=\nu(A).
         \nonumber
\end{split}
\end{equation}
The above computation shows that $g$ is a $0$-homogeneous function on $[0,\infty)$, and thus $g$ is constant. This implies that $\nu=\nu(B(0,1))\Leb^n$, and since $\nu(B(0,1))$ depends only on $n$ and $k$, the proof of the proposition is complete. % Rephrased the concluding sentence for clarity
\end{proof}

In this next proposition, we establish a representation formula for the Sobolev-Slobodeckij seminorm via slicing.

\begin{proposition}\label{restrictionW1,p}
Let \(s\in(0,1]\), $p\in [1,\infty)$ and suppose \(u\in W^{s,p}(\R^n)\). Then, for \(\gamma_{k,n}\)-almost every \(V\in\mathrm{Gr}(k,n)\) and for almost every \(z\in V^\perp\), the restriction \(u\vert_{z+V}\) of \(u\) to the affine plane \(z+V\) belongs to \(W^{s,p}(z+V)\) and there exists a constant $c_{k,n}>0$ such that 
\begin{equation}
    [u]_{W^{s,p}(\R^n)}^p = c_{k,n}\int_{V\in\mathrm{Gr}(k,n)} \int_{z\in V^\perp} [u\vert_{z+V}]_{W^{s,p}(z+V)}^p\, d\Haus^{n-k}\llcorner V^\perp(z)\, d\gamma_{k,n}(V).\nonumber
\end{equation}
% and similarly
% $$\lVert u\rVert_{L^p(\R^n)}^p=c_{k,n}\int_{V\in\mathrm{Gr}(k,n)} \int_{z\in V^\perp}\lVert u|_{z+V}\rVert_{L^p(z+V)}^p\, d\Haus^{n-k}\llcorner V^\perp(z)\, d\gamma_{k,n}(V).$$
\end{proposition}

\begin{proof}
Thanks to Proposition \ref{prop:polarhighdim} we have that 
\begin{equation}
    \begin{split}
        &[u]_{W^{s,p}(\R^n)}^p =\int \int \frac{\lvert u(x)-u(y)\rvert^p}{\lvert x-y\rvert^{n+sp}}\,dx\,dy=\int \Big(\int \lvert u(x+h)-u(x)\rvert^p\,dx\Big)\, \lvert h\rvert^{-sp-n}\,dh\\
    &\qquad=c_{k,n}\int_{V\in\mathrm{Gr}(k,n)}\int_V \Big(\int \lvert u(x+w)-u(x)\rvert^p\,dx\Big)\lvert w\rvert^{-k-sp}\, dw\, d\gamma_{k,n}(V)\\
&\qquad=c_{k,n}\int\int_{V^\perp}\Big(\int_V \int_V\frac{\lvert u(y+z+h)-u(y+z)\rvert^p}{\lvert h\rvert^{k+sp}} \,dz \,dh\Big) \,dy d\gamma_{k,n}(V),
    \nonumber
    \end{split}
\end{equation}
where in the identity in the second line we used Proposition \ref{prop:polarhighdim} with $f(h):=\int \lvert u(x+h)-u(x)\rvert^p\,dx$ and in the last line we used Fubini's theorem on the variable $x$ splitting it in $x=z+y\in V\oplus V^\perp$ and then rearranging integrals. In order to keep formulas compact in the above identities we used with abuse of notation the symbol $\int_V dz$ to denote the integral with respect to the $k$-dimensional Hausdorff measure onto $V$.

Since $u\in W^{s,p}(\R^n)$, Fubini's theorem and the above computation imply that 
\[
[u\vert_{y+V}]_{W^{s,p}(y+V)}^p := \int\int \frac{\lvert u(y+z+h)-u(y+z)\rvert^p}{\lvert h\rvert^{k+sp}}\,d\Haus^{k}\llcorner V(z)\,d\Haus^{k}\llcorner V(h) < \infty,
\]
for \(\gamma_{k,n}\)-almost every \(V\in\mathrm{Gr}(k,n)\) and for \(\Haus^{n-k}\llcorner V^\perp\)-almost every \(y\in V^\perp\).  This concludes the proof.
\end{proof}

We conclude this first subsection introducing a class of fractional Sobolev functions on rectifiable sets.

\begin{definition}\label{def:rectifiable}
    A Borel set $\Gamma$ is said to be $k$-rectifiable if there are countably many $C^1$-graphs $f_i:K_i\Subset \R^k\to \R^n$ such that 
    \begin{equation}
    \Haus^k(\Gamma\setminus \bigcup_{i\in\N} f_i(K_i))=0.
    \label{rectdeg}
    \end{equation}
    Similarly, we say that a Borel set $\Gamma$ is $(k,\mathscr{A})$-rectifiable if given a family $\mathscr{A}$ of maps there are countably many maps $f_i$ in $\mathscr{A}$ such that \eqref{rectdeg} holds.
\end{definition}

\begin{parag}[Sobolev functions on rectifiable sets]
\label{parag:def:sobolev:surface}
Let $\Gamma$ be a $k$-dimensional rectifiable set.
Given a measurable function $u:\Gamma\to\mathbb{R},$ we say that $u$ is in $W^{s,p}(\Gamma)$ if $u\in L^p(\Haus^k\llcorner \Gamma)$ and
\[
[u]_{W^{s,p}(\Gamma)}^p
:= \int \int 
\frac{\lvert u(x)-u(y)\rvert^p}{\lvert x-y\rvert^{sp + k}}
\,d\Haus^k\llcorner \Gamma(x)\,d\Haus^k\llcorner \Gamma(y)<\infty
\]
\end{parag}

\begin{proposition}\label{equivalenza}
Let $V\in \Gr(k,n)$ and let $U$ be a relatively open subset of the $k$-plane $V$. Further, let $f:U\to V^\perp$ be a Lipschitz map and denote by $\Gamma$ the graph of $f$ over $U$. Then, defined $F(z):=(z,f(z))$, for every $u\in W^{s,p}(\Gamma)$ we have that $u\circ F\in W^{s,p}(U)$. 

Viceversa, if $u\in W^{s,p}(U)$, then $u\circ \pi_V\in W^{s,p}(F(\Omega))$, where $\pi_V$ denotes the orthogonal projection onto $V$. 
\end{proposition}

\begin{proof}
Without loss of generality, suppose that $V=\mathrm{span}(e_1,\dots,e_k)$. Note that 
\[
DF(z)
= \begin{pmatrix}
\mathrm{Id}_k \\
Df(z)
\end{pmatrix}.
\]
By the area formula, we get
\[
[u]_{W^{s,p}(\Gamma)}^p
= \int\int 
\frac{\lvert u(x)-u(y)\rvert^p}{\lvert x-y\rvert^{sp + k}}
\,d\Haus^k\llcorner \Gamma(x)\,d\Haus^k\llcorner \Gamma(y)
\]
\[
=\int_U\int_U 
\frac{\lvert u\circ F(w)-u\circ F(z)\rvert^p}
{\bigl(\lvert w-z\rvert^2 + \lvert f(w)-f(z)\rvert^2\bigr)^{(sp+k)/2}}
\,JF(z)JF(w)\,dz\,dw,
\]
where $JF$ denotes the Jacobian of $F$. One immediately sees  that 
\[
JF(z)\ge \sqrt{1+\mathrm{det}\bigl(Df(z)^TDf(z)\bigr)} \ge 1,
\]
and hence
\[
\int_U\int_U 
\frac{\lvert u\circ F(w)-u\circ F(z)\rvert^p}
{\lvert w-z\rvert^{sp + k}}\,dz\,dw
\le 
\bigl(1+\mathrm{Lip}(f)^2\bigr)^{\tfrac{sp + k}{2}}
[u]_{W^{s,p}(\Gamma)}^p.
\]
This concludes the proof of the first part of the proposition.

Viceversa, denoted $\Gamma:=F(U)$ we have by the area formula 
\begin{equation}
    \begin{split}
      [u]_{W^{s,p}(F(U))}= & \int \int \frac{\lvert u\circ \pi_V(x)-u\circ \pi_V(y)\rvert^p}{\lvert x-y\rvert^{sp+k}}d\Haus^k\trace \Gamma(x)\,d\Haus^k\trace \Gamma(y)\\
      =&\int_U \int_U \frac{\lvert u(z)-u(w)\rvert^p}{\lvert x-y\rvert^{sp+k}}JF(z)JF(w)dz\,dw\leq n(1+\Lip(f)^2)[u]_{W^{s,p}(U)}. 
      \nonumber
    \end{split}
\end{equation}
This concludes the proof.
\end{proof}

\subsection{Tangency sets and commutators}

Being a paper devoted to Frobenius' theorem we need to formally introduce distributions, Lie brackets of vector fields and tangency sets.

\begin{parag}[Distributions of $\boldsymbol k$-planes]
\label{def:distrib}
Let $1\le k\le n$.
A \emph{distribution of $k$-planes} on the open set $\Omega$ 
in $\R^n$ is a map $V$ that associates to every $x\in\Omega$ 
a $k$-dimensional plane $V(x)$ in $\R^n$, that is, 
a map from $\Omega$ to the Grassmannian $\Gr(k,n)$.

We say that a \emph{system of} $k$ \emph{vector fields}  $\mathcal{X}:=\{X_1,\dots, X_k\}$ 
\emph{spans} $V$ if for every $x\in \Omega$ one has
\[
V(x)
=\Span(\mathcal{X}(x)) :=\Span\big\{ X_1(x),\dots,X_k(x) \big\}
\, .
\] 
We say that the distribution $V$ is
of class $C^r$, with $r=0,1,\dots,\infty$,
if it is \emph{locally} spanned by $\{X_1,\dots, X_k\}$,
where the vector fields $X_i$ are of class $C^r$.%
%
%\footnoteb{If $\Omega$ is simply connected, or if 
%$V$ is a distribution of \emph{oriented} $k$-planes, 
%then $V$ is \emph{globally} spanned by a 
%simple $k$-vector field $v$ of class $C^r$.
%This does not imply that $v$ can 
%be globally written as product of vector fields $v_i$
%of class $C^r$, or even continuous.}

Given $h=1,\dots,k$ we say that a system of $h$ vector fields $\mathcal{W}$ on $\Omega$
is \emph{tangent} to $V$ if $\Span(\mathcal{W}(x)) \subseteq V(x)$ for every~$x$
(simply $\mathcal{W}(x)\in V(x)$ when $h=1$).
\end{parag}

\begin{parag}[Lie brackets of vector fields]
\label{def:lie}
Recall that given two vector fields $X$, $X'$ on $\Omega$ of class $C^1$, 
the Lie bracket $[X,X']$ is the vector field on $\Omega$ 
defined by
\[
[X,X'](x)
:= \frac{\bd X}{\bd X'}(x) - \frac{\bd X'}{\bd X}(x)
= \text{D}_x X \, (X'(x)) - \text{D}_x X' \, (X(x))
\, , 
\]
where $\text{D}_x X$ and $\text{D}_x X'$ stand for the differentials of $X$ and $X'$
at the point $x$, viewed as linear maps from $\R^n$ into itself.
% Given a system of $k$ smooth vector fields $\mathcal{X}:=\{X_1,\dots, X_k\}$  on $\Omega$, we say that the the vector field $Y$ is an elementary commutator of $\mathcal{X}$ and has length or degree $N$ if there exist $Y_1,\ldots,Y_N\in\mathcal{X}$ such that
% $$Y=[Y_1,[\ldots,[Y_{N-1},Y_N]]].$$
\end{parag}

\begin{parag}[Involutivity of a distribution $\boldsymbol{V}$
and the set $\boldsymbol{N(V)}$]
\label{def:inv}
Let $V$ be a distribution of $k$-planes 
of class $C^1$ on the open set $\Omega$ in $\R^n$.

We say that $V$ is \emph{involutive} at a point $x\in\Omega$ if
for every pair of vector fields $X$, $X'$ of class $C^1$ which
are tangent to $V$ 
the commutator $[X,X'](x)$ belongs to $V(x)$.
We say that $V$ is involutive if it is involutive 
at every point of~$\Omega$.

The collection of all points $x$ where $V$ is not involutive is called
the \emph{non-involutivity set} of $V$ and denoted by $N(V)$.
Note that this set is open. 
\end{parag}

\begin{parag}[Tangency sets of a surface to a distribution]
\label{def:tan}
Let $V$ be a distribution of $k$-planes of class $C^1$ on the open set $\Omega$ in $\R^n$ and $S\subseteq \Omega$ be a $k$-dimensional manifold of class $C^1$. 
We say that $x\in S$ is a \emph{tangency point} of $S$ with respect to $V$ if and only if $\Tan(S,x)=V(x)$, where here $\Tan(S,x)$ denotes the classical tangent of the surface $S$ at $x$. 
The set of such points is called the \emph{tangency set} of $S$ with respect to $V$ and denoted by:
$$\tau(S,V):=\{x\in S:\Tan(S,x)=V(x)\}.$$
\end{parag}

\subsection{Cantor-type sets, their dimension and boundary regularity}\label{cantortypesets}

The counterexamples required for the sharpness part of
Theorem~\ref{main:intro:frobenius} relies on two complementary ingredients:  
the construction of a compact set \(C\) whose indicator belongs to
\(W^{s,1}(\Omega)\) and  
the construction of a \(k\)-dimensional graph tangent to the distribution
\(V\) on \(C\).
In this subsection we address the first task by constructing
the aforementioned Cantor-type set.

Throughout this subsection \(\Omega\subset\R^{k}\) denotes a fixed open,
bounded set.
The procedure is classical in spirit: we iteratively construct a lattice of
axis-parallel nested cubes that are separated by strips of width dictated by some sequence \(\{\rho_j\}_{j\ge 1}\).
The parameters are tuned so that the set
\[
  C \;=\;\bigcap_{j=1}^{\infty}\,
        \bigcup_{Q\in\Delta_j} Q,
\]
where \(\Delta_j\) is the family of surviving cubes at the \(j\)-th
generation, satisfies
\(\mathscr L^{k}(C)>0\) and
\(\mathbb 1_{C}\in W^{s,p}(\Omega)\).
Sections~\ref{definizionecubani}--\ref{compatto}
detail the construction and give the precise quantitative bounds
\begin{equation*}
  \| \mathbb 1_{C}\|_{W^{s,p}(\Omega)}
  \;\le\;
  c(k,s,p,\Omega)\Bigl(
     1+\sum_{j=1}^{\infty} 2^{Bj}\rho_{j}^{\,1-sp}
  \Bigr),
\end{equation*}
where \(B\in\N\) is a fixed branching parameter and
\(c=c(k,s,p,\Omega)\) is an explicit constant.
The choice of \(\{\rho_j\}\) will later be coupled with the geometric
oscillation of the tangent graphs, ensuring that the pair
compact set and graph realizes the regularity predicted by
Theorem~\ref{lusin:intro}.

% In order to prove the sharpness of Theorem \ref{main:intro:frobenius} we will need to construct counterexamples. In order to do this we will need to construct a compact set whose indicator function will have fractional Sobolev regularity and on such sets we will construct graphs that will be tangent to the given distribution of planes. In this subsection we will construct such compact sets in the form of Cantor-type sets. Throughout this section, $\Omega$ will be a fixed open and bounded set in $\R^k$. In this section, we construct a Cantor-type compact set whose indicator function is in $W^{s,p}(\Omega)$. In order to construct such a set, we first introduce a family of cubes in $\Omega$. 

\begin{definition}\label{cubes} Let $\varepsilon_0\in(0,1/10)$. Denote by $Q({p},s)$ the closed cube centered at ${p}$ with side length $s$. Viceversa, given a closed cube $Q$ in $\R^k$ we denote by ${ c}(Q)$ and $\ell(Q)$ the center and the side-length of $Q$ respectively. 
For any $s>0$, we set the lattice
$$L(s,\Omega):=(s\Z)^{k}\cap \{{ p}\in\Omega:\,Q({ p},s)\subset\Omega\},$$
and we denote by $\delta_0:=\delta_0(\varepsilon_0,\Omega)>0$ the supremum of those $0<\delta<1/25$ for which
\[
\Leb^{k}\left(\bigcup_{{ p}\in L(\delta,\Omega)} Q({ p},\delta)\right)\ge (1-\varepsilon_0)\Leb^{k}(\Omega)\,.
\]
In the following, we let $\Delta_0(\delta):=\{Q(p,\delta):p\in L(\delta,\Omega)\}$.
\end{definition}

Roughly speaking, we aim at working with a lattice \(L(\delta,\Omega)\) whose mesh $\delta<\delta_0$ is still guaranteeing an acceptable approximation of \(\Omega\). From now on, \(\delta_0\) is considered as a fixed parameter of the problem and will not be recalled.  Since $\Omega\subset\R^k$ is fixed as well, from now on we will drop further dependence on it.

\begin{definition}\label{definizionecubani}
Let $\delta<\delta_0$ and $B\in\N$. In the following ,with the symbol $\boldsymbol{\rho}$, we will denote a sequence of positive numbers $\boldsymbol{\rho}:=(\rho_j)_{j\in\N}$ such that 
\begin{equation}
    \sum_{j\in \N}2^{Bj}\rho_{j+1}\leq \delta.
    \label{ipotesirho1}
\end{equation}
We now construct a family of cubes $\Delta=\Delta(\delta, B)$ that is the union of subfamilies $\{\Delta_i\}_{i\in\N}$ of disjoint cubes with the same side length $r_i$ and having centers in a certain discrete set $L_i\subseteq \Omega$. This construction will be performed inductively.

\smallskip

Let $r_0:=\delta$, $L_1:=L(\delta,\Omega)$ and  $r_1:=\delta-\rho_1$.
We define the first layer $\Delta_1$ to be the family of closed cubes centered at $L_1$  with side $r_1$. At the $i^{\rm th}$ step, we let
$$
L_i:=\bigcup_{Q\in \Delta_{i-1}} \mathfrak{c}(Q)+\frac 14r_{i-1}\mathbb{1}+\Big(\frac 12r_{i-1}\mathbb{Z}\Big)^{k}\cap Q,
$$
where $\mathbb{1}:=(1,\ldots,1)\in \R^k$. Finally, the elements of $\Delta_i$ are defined to be those cubes $Q(p,r_i)$ with center $p\in L_i$ and side length
$r_i:=2^{-B}r_{i-1}-\rho_i$.

\smallskip

The family of cubes $\Delta=\bigcup_{i\ge 1}\Delta_i$ has the following ``genealogical'' property. If $C\in\Delta_{i+1}$, then there exists a unique $\mathfrak{f}\in \Delta_i$ such that $C\subset \mathfrak{f}(C)$.
\end{definition}

\begin{parag}[Restrictions on the decay of $\boldsymbol{\rho}$]\label{decayofrho}

One immediately sees that the very definition of the $r_i$s yields 
$$r_{i}=\frac{\delta-\sum_{j=1}^i2^{Bj}\rho_j}{2^{Bi}}.$$
Thanks to \eqref{ipotesirho1}, we see that $r_i>0$ for every $i\in \N$ and thanks to our choice of $r_i$ we have 
\begin{equation}
r_i\leq \delta 2^{-Bi}\qquad\text{and}\qquad r_i\leq  2^{-B}r_{i-1}.
    \label{stimer_i}
\end{equation}
In order to simplify the computations in the following we impose further conditions on $\boldsymbol{\rho}$. More specifically we suppose that 
\begin{equation}
    \{\rho_\iota\}_{\iota\in\N} \text{ and }\{\rho_{\iota+1}^{-1}r_{\iota+1}r_\iota\}_{\iota\in \N}\text{ are decreasing and }\sum_{\iota\in \N}\rho_{\iota+1}^{-1}r_{\iota+1}r_\iota<\infty.
    \label{ipotesisurho}
\end{equation}
% In the following, it will be convenient to let
% \begin{equation}
%     \mathfrak{N}:=\sum_{j\geq 2}(8k^{\frac{3}{2}}\rho_{j+1}^{-1}r_{j+1}+1)r_j.
% \label{denNNN}
% \end{equation}
\end{parag}

\begin{parag}[Cantor-type sets associated to $\boldsymbol{\rho}$]
\label{defCompatto}
In the following we let 
$\mathfrak{C}:=\mathfrak{C}(\delta,\boldsymbol{\rho},B,\Omega)$ be the compact set
associated with the cubes $\Delta$ and defined as
$$\mathfrak{C}:=\bigcap_{j\in \N} \bigcup_{Q\in \Delta_j} Q.$$
In the following it will be convenient to set $\mathfrak{C}_j   :=\bigcup_{Q\in \Delta_j} Q$.
\end{parag}

\begin{proposition} \label{compatto}
Let $s\in (0,1)$ and suppose that $\boldsymbol{\rho}$ is a sequence such that 
    $$\sum_{j\in \N}2^{Bj}\rho_{j}< \delta.$$
Then, $\mathscr{L}^k(\mathfrak{C})>0$ and if $\sum_{j\in \N}2^{Bj}\rho_{j}^{1-s}<\infty$, we have
$$\text{$\mathbb{1}_{\mathfrak{C}}\in W^{s,1}(\Omega)$ and $\lVert \mathbb{1}_{\mathfrak{C}}\rVert_{W^{s,1}(\Omega)}\lesssim_{k,B,s,\Omega}\mathscr{L}^k(\Omega)+\sum_{j\in \N}2^{Bj}\rho_j^{1-s}$,}$$
where $\lesssim_{k,B,s,\Omega}$ means that the inequality holds up to constants depending on $k,B,s$ and $\Omega$.
\end{proposition}

\begin{proof}
    One immediately sees  that 
$$\mathscr{L}^{k}\Big(\bigcup_{Q\in \Delta_i}Q\Big)=\mathrm{Card}(L_1)\Big(\delta-\sum_{\iota=0}^i2^{B\iota}\rho_\iota\Big)^{k}.$$
Therefore, thanks to the continuity of the measure from above we infer that
$$\mathscr{L}^{k}(\mathfrak{C})=\mathrm{Card}(L_1)\Big(\delta-\sum_{\iota\in\N}2^{B\iota}\rho_\iota\Big)^{k}>0.$$
 Clearly, $\lVert \mathbb{1}_{\mathfrak{C}}\rVert_{L^1(\Omega)}\leq \mathscr{L}^k(\Omega)$. Therefore, in order to prove that 
$\mathbb{1}_{\mathfrak{C}}\in W^{s,1}(\Omega)$ we just need to estimate the seminorm
$[\mathbb{1}_\mathfrak{C}]_{W^{s,1}(\Omega)}$. One immediately sees  that 
$$[\mathbb{1}_\mathfrak{C}]_{W^{s,1}(\Omega)}=2\int_{\mathfrak{C}^c}\int_{\mathfrak{C}}\frac{dxdy}{\lvert x-y\rvert^{k+s}},$$
and thus
\begin{equation}
\begin{split}
[\mathbb{1}_\mathfrak{C}]_{W^{s,1}(\Omega)}=&2\int_{ \mathfrak{C}_1^c}\int_{\mathfrak{C}}\frac{dxdy}{\lvert x-y\rvert^{k+s}}\\
    =&2 \int_{\mathfrak{C}_1^c}\int_{\mathfrak{C}}\frac{dxdy}{\lvert x-y\rvert^{k+s}}+2\sum_{i=1}^\infty\int_{\mathfrak{C}_{i}\setminus \mathfrak{C}_{i+1}}\int_{\mathfrak{C}}\frac{dxdy}{\lvert x-y\rvert^{k+s}},
\end{split}
\end{equation}
where the sets $\mathfrak{C}_j$ were introduced in \S \cref{defCompatto}.
Fix $i\in \N$ and let $y\in \mathfrak{C}_{i}\setminus \mathfrak{C}_{i+1}$. Then, we have
\begin{equation}
    \begin{split}
&\qquad\qquad\qquad\int_{\mathfrak{C}} \frac{dx}{\lvert x-y\rvert^{k+s}}=\int_{\mathrm{dist}(y,\mathfrak{C})}^\infty \frac{\mathscr{H}^{k-1}(\partial B(y,t)\cap \mathfrak{C})}{t^{k+s}}dt\\
=&\Big[\frac{\mathscr{L}^{k}( B(y,t)\cap \mathfrak{C})}{t^{k+s}} \Big]_{\mathrm{dist}(x,\mathfrak{C})}^\infty-(k+s)\int_{\mathrm{dist}(y,\mathfrak{C})}^\infty\frac{\mathscr{L}^{k}( B(y,t)\cap\mathfrak{C})}{t^{k+1+s}}dt\\
&\qquad\qquad\qquad=-(k+s)\int_{\mathrm{dist}(y,\mathfrak{C})}^\infty\frac{\mathscr{L}^{k}( B(y,t)\cap \mathfrak{C})}{t^{k+1+s}}dt,
\nonumber
    \end{split}
\end{equation}
% where the last identity comes from the fact that $$\mathscr{L}^{k}( B(y,\mathrm{dist}(y,\mathfrak{C}))\cap \mathfrak{C}_{i}\setminus \mathfrak{C}_{i+1})=0.$$
From the above computation, we infer that 
\begin{equation}
    \begin{split}
        &\qquad\quad\int_{\mathfrak{C}} \frac{dx}{\lvert x-y\rvert^{k+s}}\leq -(k+s)\mathscr{L}^k(B(0,1))\int_{\mathrm{dist}(y,\mathfrak{C})}^\infty\frac{dt}{t^{1+s}}\\
        &=(k+s)s\mathscr{L}^k(B(0,1))\mathrm{dist}(y,\mathfrak{C})^{-s}
       \leq (k+s)s\mathscr{L}^k(B(0,1))\rho_{i+1}^{-s},
        \nonumber
    \end{split}
\end{equation}
where the first identity in the second line follows from the definition of the $r_i$s and
and the last inequality comes from the fact that $y\in \mathfrak{C}_{i}\setminus \mathfrak{C}_{i+1}$. The volume of $\mathfrak{C}_{i}\setminus \mathfrak{C}_{i+1}$ can be estimated as follows. Recall that the cardinality of each $\Delta_i$ is $\mathrm{Card}(L_1)2^{Bki}$, and since the cubes of $\Delta_i$ have sidelength $r_i$ we have 
\begin{equation}
    \begin{split}
       & \mathscr{L}^k(\mathfrak{C}_{i}\setminus \mathfrak{C}_{i+1})=\mathrm{Card}(L_1)2^{Bki}r_i^k-\mathrm{Card}(L_1)2^{Bk(i+1)}r_{i+1}^k\\
        =&\mathrm{Card}(L_1)2^{Bki}(r_i^k-2^{Bk}r_{i+1}^k)=\mathrm{Card}(L_1)2^{Bki}(r_i^k-2^{Bk}(2^{-B}r_i-\rho_i)^k)\\
        =&\mathrm{Card}(L_1)2^{Bki}r_i^k(1-(1-2^{B}r_i^{-1}\rho_i)^k)\leq 2^Bk\mathrm{Card}(L_1)2^{Bki}r_i^{k-1}\rho_{i},
        \label{stimamassadiffCC}
    \end{split}
\end{equation}
where the last inequality is a consequence of the Bernoulli's inequality. The above computation shows in particular that
\begin{equation}
    \begin{split}
    &\int_{\mathfrak{C}_{i}\setminus \mathfrak{C}_{i+1}}\int_{\mathfrak{C}}\frac{dxdy}{\lvert x-y\rvert^{k+s}}\leq \mathscr{L}^k(\mathfrak{C}_{i}\setminus \mathfrak{C}_{i+1})(k+s)s\mathscr{L}^k(B(0,1))\rho_{i+1}^{-s}\\
    &\qquad \leq2^Bk\mathrm{Card}(L_1)2^{Bki}r_i^{k-1}\rho_{i}\,\cdot\,(k+s)s\mathscr{L}^k(B(0,1))\rho_{i+1}^{-s}\\
    &\qquad\qquad\leq 2^k\delta^{k-1}k(k+1)\mathrm{Card}(L_1)2^{B(i+1)}\rho_i\rho_{i+1}^{-s}\\
     &\qquad\qquad\leq2^k\delta^{k-1}k(k+1)\mathrm{Card}(L_1)2^{B(i+1)}\rho_i^{1-s},
    \nonumber
    \end{split}
\end{equation}
where the last inequality comes from the fact that $\rho_i$ is assumed to be decreasing.
Summing up, we infer that 
\begin{equation}
    \begin{split}
        &[\mathbb{1}_\mathfrak{C}]_{W^{s,1}(\Omega)}\leq 2 \int_{\mathfrak{C}_1^c}\int_{\mathfrak{C}}\frac{dxdy}{\lvert x-y\rvert^{k+s}}\\
&\qquad\qquad\qquad\qquad\qquad+k(k+1)\mathrm{Card}(L_1)\sum_{i=1}^\infty2^{B(i+1)}\rho_{i}^{1-s}<\infty,
    \nonumber
    \end{split}
\end{equation}
where the last inequality comes from the fact that because of the choice of $\delta$ we have $2^k\delta^{k-1}\leq 1$. 
This concludes the proof. 
\end{proof}

\begin{remark}
    Note that, if 
    $\sum_{j\in \N}2^{Bj}\rho_{j}=\delta$, then $\mathscr{L}^k(\mathfrak{C})=0$. 
\end{remark}

The following result will be employed to construct the tangency set in the $C^{1,1}$-regime, see Theorem \ref{teoremaC1,1}.

\begin{proposition}\label{dimensione}
    Let $0<\lambda<1$ and $\rho_i:=\delta(\lambda-1)\lambda^{i}2^{-Bi}$. Then, $\mathfrak{C}$ is a self-similar fractal in the sense of Hutchinson, see \cite{hutchinson}, and
    $\mathfrak{C}$ has dimension $\frac{B}{B-\log_2\lambda}k=:\mathfrak{d}$ and there exists a constant $c\geq 1$ such that 
    $$c^{-1}\leq \liminf_{r\to 0}\frac{\Haus^\mathfrak{d}\llcorner \mathfrak C(B(x,r))}{r^\mathfrak{d}}\leq \limsup_{r\to 0}\frac{\Haus^\mathfrak{d}\llcorner \mathfrak C(B(x,r))}{r^\mathfrak{d}}\leq c.$$
\end{proposition}

\begin{proof}
    It is easily checked that 
    the set $\mathfrak{C}$ is the fixed point of  $2^{Nk}$ affine transformations with Lipschitz constant $\lambda2^{-B}$. Therefore \cite[Theorem (3)]{hutchinson} yields the claimed result. 
\end{proof}

\section{Locality of the exterior differential operator}

\subsection{Stokes Theorem for rough forms}

In this section we prove a locality result for the exterior differential in the plane. As a first step, we prove a general Stokes theorem for rough forms.

\begin{parag}[Distributional exterior derivative of $k$-forms]
Let $\omega$ be a continuous $k$-form. We define its distributional exterior derivative $\mathrm{d}\omega$ by duality, namely, for any $(k+1)$-current $T$ of the form $T=\tau\,\mathscr{L}^n$, where $\tau$ is a $(k+1)$-vector field in $C^\infty_c$, we require
$$\langle\partial T,\omega\rangle=\langle T,\mathrm{d}\omega\rangle.$$
\end{parag}

\begin{proposition}\label{divergencethm}
Suppose $\Omega\subset\R^n$ is a bounded open set of finite perimeter and let $\omega$ be a continuous $(n-1)$-form with $\mathrm{d}\omega\in L^1_{\mathrm{loc}}(\R^n)$, where $\mathrm{d}\omega$ must be understood in the distributional sense. Then
$$\int_\Omega \mathrm{d}\omega=\int_{\partial^*\Omega}\omega:=\int\langle \mathfrak{n}\wedge\omega,e_1\wedge\cdots\wedge e_n\rangle\,d\Haus^{n-1}\llcorner \partial^*\Omega,$$
where $\partial^*\Omega$ denotes the reduced boundary of $\Omega$ and where $\mathfrak{n}$ is the 1-form corresponding via polarity, see \cite[\S 1.7.1]{Federer1996GeometricTheory}, to the measure-theoretic unit normal to $\partial^*\Omega$.
\end{proposition}

\begin{proof}
Let $T$ be the natural normal $n$-current associated with $\Omega$, defined by
$$T:=\mathbb{1}_\Omega\, e_1\wedge\cdots\wedge e_n.$$ 
Fix $\varepsilon>0$ and let $\rho$ be a standard radially symmetric convolution kernel. Define $\rho_\varepsilon(x):=\varepsilon^{-n}\rho(x/\varepsilon)$. By the definition of the boundary operator for currents, for every $\varepsilon,\delta>0$ we have
\begin{equation}
\langle\partial T,(\omega*\rho_\delta)*\rho_\varepsilon\rangle=\langle\partial (T*\rho_\varepsilon),\omega*\rho_\delta\rangle=\langle T*\rho_\varepsilon,\mathrm{d}(\omega*\rho_\delta)\rangle=\langle T,\mathrm{d}(\omega*\rho_\delta)*\rho_\varepsilon\rangle,
\label{eq:idconv}
\end{equation}
where we repeatedly use the symmetry of $\rho$ and the standard properties of convolution and the boundary operator. We now claim that
\begin{equation}
\langle T,\mathrm{d}(\omega*\rho_\delta)*\rho_\varepsilon\rangle=\langle T,\mathrm{d}\omega*(\rho_\delta*\rho_\varepsilon)\rangle.
\label{id:distro}
\end{equation}
Assuming \eqref{id:distro}, it follows from \eqref{eq:idconv} and the associativity of convolution that for every $\delta,\varepsilon>0$
\begin{equation}
\langle\partial T,\omega*(\rho_\delta*\rho_\varepsilon)\rangle=\langle T,\mathrm{d}\omega*(\rho_\delta*\rho_\varepsilon)\rangle.
\label{eq:idconv2}
\end{equation}
Since $\omega$ is continuous and $\partial T$ is represented by a compactly supported vector-valued finite measure, choosing $\delta=\varepsilon$ and letting $\varepsilon\to0$ yields
$$\langle\partial T,\omega\rangle=\lim_{\varepsilon\to0}\langle T,\mathrm{d}\omega*(\rho_\varepsilon*\rho_\varepsilon)\rangle.$$
Because $\mathrm{d}\omega\in L^1_{\mathrm{loc}}(\R^n)$ and by standard properties of mollifiers, $\mathrm{d}\omega*(\rho_\varepsilon*\rho_\varepsilon)$ converges to $\mathrm{d}\omega$ in $L^1_{\mathrm{loc}}(\R^n)$, we conclude that
$$\langle\partial T,\omega\rangle=\langle T,\mathrm{d}\omega\rangle.$$
In addition, by definition of $T$, we have
\begin{equation}
\begin{split}
\partial T &= e_1\wedge\cdots\wedge e_n\llcorner d\mathbb{1}_\Omega = e_1\wedge\cdots\wedge e_n\llcorner \beta(D\mathbb{1}_\Omega)\\
&=(e_1\wedge\cdots\wedge e_n\llcorner \mathfrak{n})\Haus^{n-1}\llcorner \partial^*\Omega,
\end{split}
\end{equation}
where $\beta$ is the polarity map defined in \cite[\S 1.7.1]{Federer1996GeometricTheory} and $\mathfrak{n}:=\beta(D\mathbb{1}_\Omega)$. Hence,
\begin{equation}
\begin{split}
\int_\Omega \mathrm{d}\omega &=\langle\partial T,\omega\rangle=\int\langle \omega,e_1\wedge\cdots\wedge e_n\llcorner \mathfrak{n}\rangle\,d\Haus^{n-1}\llcorner \partial^*\Omega\\
&=\int\langle \mathfrak{n}\wedge\omega,e_1\wedge\cdots\wedge e_n\rangle\,d\Haus^{n-1}\llcorner \partial^*\Omega=\int_{\partial^*\Omega}\omega,
\end{split}
\end{equation}
which completes the proof of the proposition given \eqref{id:distro}.

We now prove \eqref{id:distro}. It suffices to show that for every smooth, compactly supported normal current $S$ and any test function $g$, the equality
$$\langle S,\mathrm{d}(\omega*g)\rangle=\langle S,\mathrm{d}\omega*g\rangle$$
holds. By the definition of the distributional derivative we have
$$\langle\partial(S*\hat{g}),\omega\rangle=\langle S*\hat{g},\mathrm{d}\omega\rangle,$$
where $\hat{g}(x)=g(-x)$. Applying Fubini's theorem and using the smoothness and compact support of $S$ and $g$, we deduce that
$$\langle S*\hat{g},\mathrm{d}\omega\rangle=\langle S,\mathrm{d}\omega*g\rangle,$$
noting that $\mathrm{d}\omega\in L^1_{\mathrm{loc}}(\R^n)$. The identity
$$\langle\partial(S*\hat{g}),\omega\rangle=\langle (\partial S)*\hat{g},\omega\rangle$$
is evident for smooth $\omega$, and by approximating $\omega$ with smooth forms the equality extends by density. Similarly, one verifies that
$\langle (\partial S)*\hat{g},\omega\rangle=\langle \partial S,\omega*g\rangle$.
This completes the proof of \eqref{id:distro} and hence of the proposition.
\end{proof}

\subsection{Slicing superdensity sets and locality of the exterior differential}

The following is a technical proposition. It's objective is first of all to produce at almost every point of a set with fractional boundary $E$ a sequence of rectangles of bounded eccentricity such that the the boundary of such rectangles meets $E$ in a set of quantifiable big length. Secondly we want to prove that on such rectangles the variation of fractional Sobolev functions is controlled by slicing-type estimates.

\begin{proposition}\label{prop:rettangoli}
Let $0<\varepsilon<1/10$,
$0<s<1$, and let $B$ be a ball in $\R^n$. Suppose that $E\subset\R^2$ is a Lebesgue measurable set with $\mathbb{1}_E\in W^{s,1}(B)$. Assume further that $g\in W^{\alpha,q}(B)$.
Then, there exists a vector $v\in\mathbb{S}^1$ such that fixed $b<s$ and $\tilde{b}<\alpha$, for $\Leb^2$-almost every $x\in E$ one may find a sequence $r_i\to0$ and a constant $\psi=\psi(g,x)$ such that
the rectangle $\mathscr{P}_i(x)$ in $\R^2$, with sides $\mathbb{L}_{1,3}^i$ and $\mathbb{L}_{2,4}^i$ parallel to $v$ and $v^\perp$ respectively, satisfies the following properties.
\begin{itemize}
\item[(i)] If $\mathfrak{b}_i(x)$ denotes the barycenter of $\mathscr{P}_i(x)$, then 
$$|\mathfrak{b}_i(x)-x|\leq 2\varepsilon\,r_i;$$
\item[(ii)] If $\ell_{1,i}$ and $\ell_{2,i}$ denote the lengths of $\mathbb{L}_{1,3}^i$ and $\mathbb{L}_{2,4}^i$ respectively, then 
$$|\ell_{j,i}-2r_i|\leq 4\varepsilon\,r_i\qquad\text{for every $j=1,2$ and every $i\in\N$};$$
\item[(iii)]For every $i\in\N$ we have
$$\Haus^1\bigl(\partial\mathscr{P}_i(x)\setminus E\bigr)\leq 4\varepsilon\,r_i^{\,1+\frac{(1-i^{-1})s}{1-s}};$$
\item[(iv)] For every $\kappa=1,\ldots,4$ and every $i\in\N$ we have
$$\int\int \frac{\lvert g(z)-g(y)\rvert^{q}}{\lvert z-y\rvert^{1+\alpha q}}d\Haus^1\llcorner \mathbb{L}_\kappa^i(z)\,d\Haus^1\llcorner \mathbb{L}_\kappa^i(y)\leq \psi(g,x)r_i;$$
\item[(v)] If 
$\alpha q\leq 1$ there are vertices $\mathfrak{c}_{i}^\kappa\in E$ of $\mathscr{P}_i(x)$ such that for every $\kappa=1,\ldots,4$ we have 
$$\Big(\fint \lvert g(z+\mathfrak c_{i}^\kappa)-g(\mathfrak c_{i}^\kappa)\rvert^{q^*} d\Haus^1\llcorner \mathbb{L}^i_\kappa(z)\Big)^\frac{1}{q^*}\leq \varepsilon r_i^{\widetilde{b}}.$$
\end{itemize}
\end{proposition}

\begin{proof}
By \cite[Theorem 5.4]{HGfractionalsobolev}, we can extend $\mathbb{1}_E$ to a function $u\in W^{s,1}(\R^n)$. In what follows we denote by $J$ the standard $2\times 2$ symplectic matrix. 
By Proposition~\ref{restrictionW1,p}, for $\Haus^1$-almost every $v\in\mathbb{S}^1$ and $\Leb^1$-almost every $t\in\R$ we have
\begin{equation}
u\bigl|_{tJv+\R v}\in W^{s,p}(\R)\qquad\text{and}\qquad g\bigl|_{tJv+\R v}\in W^{\alpha,q}(\R),
\label{t1t2pm}
\end{equation}
where $u\bigl|_{tJv+\R v}$ and $g\bigl|_{tJv+\R v}$ denote the indicator function of the restriction of $u$ and the restriction of $g$ onto the line $tJv+\R v$ respectively. 
For every $e\in \mathbb S^1$ and $w\in \R^2$, define 
\begin{equation}
        \gimel_b(w,r\,;e):= \sup_{\tau\leq r}\tau^{- b}\Big(\fint\lvert u(z)-u(w)\rvert^{1^*} d\Haus^1\llcorner (w+\ell_e\cap B(w,10\tau))(z)\Big)^\frac{1}{1^*},
        \nonumber
        \end{equation}
        and
        \begin{equation*}
\daleth(w,r\,;e) := \sup_{\tau\leq r} \frac{1}{20\tau} \int_{-10\tau}^{10\tau} \int_{-10\tau}^{10\tau} \frac{\lvert g(w+he)-g(w+t e)\rvert^{q}}{\lvert h-t\rvert^{1+\alpha q}}\,dh\,dt
\end{equation*}
where $1^*=1/(1-s)$. In case $\alpha q\leq 1$ we also let 
$$\beth_{\tilde{b}}(w,r; e):=\sup_{\tau<r}\tau^{-\tilde b}\Big(\fint\lvert g(z)-g(w)\rvert^{q^*} d\Haus^1\llcorner (w+\ell_e\cap B(w,10\tau))(z)\Big)^\frac{1}{q^*},$$
where $q^*=q/(1-\alpha q)$.
Thanks to Proposition \ref{dorronsorosuper-density}, for $\Haus^1$-almost every $v\in \mathbb S^1$, $\Leb^1$-almost every $t\in \R$ and $\Haus^1$-almost every $w\in E\cap tJv+\R v$ we have 
$$\lim_{r\to 0}\gimel_b(w,r\,;v)=0\qquad\text{and}\qquad\lim_{r\to 0}\beth_{\tilde b}(w,r\,;v)=0\text{ provided }\alpha q\leq 1.$$
Hence, by Fubini's theorem, we conclude that for $\Haus^1$-almost every $v\in \mathbb{S}^1$ and for $\Leb^2$-almost every $w\in E$ we have 
$$\lim_{r\to 0}\gimel_b(w,r\,;v)=0\qquad\text{and}\qquad\lim_{r\to 0}\beth_{\tilde{b}}(w,r\,;v)=0\text{ provided }\alpha q\leq 1.$$
Let us analyze the behavior of $\daleth$. One immediately sees arguing as above that for $\Leb^2$-almost every $w\in \R^2$ we have 
$$\lim_{r\to 0}\daleth(w,r;e)\leq \int \frac{\lvert g(w+he)-g(w)\rvert^{q}}{\lvert h\rvert^{1+\alpha q}}\,dh=:\psi(g,w)<\infty.$$
Thanks to Severini-Egoroff's theorem, for every $\varepsilon>0$ there exists a compact set $K_0=K_0(\varepsilon,e)\subseteq E$ with $\Leb^2(E\setminus K_0)\leq \varepsilon\Leb^2(E)/3$ and a $\delta>0$ such that if $r\leq \delta$ then
$$\daleth(w,r;e)\leq 2\psi(g,x)\qquad\text{for every }w\in K_0.$$
This implies that 
\begin{equation}
    \int_{K_0} \daleth(w,r;e)d\Leb^2(w)\leq 2\int \int \psi(he+tJe)dh dt=2\int [g\lvert_{ t Je+\ell_e}]_{W^{\alpha,q}( t Je+\ell_e)}^qdt.
    \nonumber
\end{equation}
Thanks to Proposition \ref{restrictionW1,p} this implies that for $\Haus^1$-almost every $e\in \mathbb{S}^1$ we have that $\daleth(w,r;e)\in L^1(K_0)$ where we recall once again that $K_0$ depends also on $e$. If we perform the same argument with the set $JE$, we see that for $\Haus^1$-almost every $v\in \mathbb{S}^1$  there exists a compact set $K_1:=K_1(\varepsilon,v)\subseteq K_0$ such that $\Leb^2(E\setminus K_1)\leq 2\varepsilon\Leb^2(E)/3$ 
$$\daleth(w,r;v)\in L^1(K_1)\qquad\text{and}\qquad\daleth(w,r;Jv)\in L^1(K_1).$$
In turn this implies that
$$\lim_{r\to 0}\fint_{B(w,r)}\lvert \daleth(z,r;e_\iota)-\daleth(w,r;e_\iota)\rvert dz =0\text{ for $\Leb^2$-almost every $w\in K_1$ and $\iota=1,2$},$$
where $e_1=v$ and $e_2=Jv$.  

Furthermore, arguing as above by Severini-Egoroff's theorem we know that for every $\varepsilon>0$ there exists a compact set $K\subseteq K_1$ with $\Leb^2(E\setminus K)\leq \varepsilon\Leb^2(E)$ such that for every $\eta>0$ there exists $\delta>0$ such that if $r<\delta$ then 
\begin{equation}
    \begin{split}
        &\qquad\qquad\gimel_b(w,r\,;e_\iota)\leq \eta\quad\text{for $\Leb^2$-almost every $w\in K$ and $\iota=1,2$},\\
        &\beth_{\tilde b}(w,r\,;e_\iota)\leq \eta\quad\text{for $\Leb^2$-almost every $w\in K$ and $\iota=1,2$ provided $\alpha q\leq 1$},
        \label{smellozzola1}
    \end{split}
\end{equation}
and 
\begin{equation}
    \begin{split}
        \fint_{B(w,r)}\lvert \daleth(z,r;e_\iota) -\daleth(w,r;e_\iota)\rvert dz \leq \eta\quad\text{for $\Leb^2$-almost every $w\in K$ and $\iota=1,2$}.
        \nonumber
    \end{split}
\end{equation}
Select a density point $p$ of $K$ and note that for every $\eta>0$ there is $r_0=r_0(\eta)>0$ such that  whenever $r<r_0$
there are points $w_{1,r}, w_{2,r}\in K$ such that 
\begin{equation}
    \lvert p+(-1)^j(v+Jv)r-w_{j,r}\rvert\leq \eta r \qquad\text{ for }j=1,2.
    \label{smellozzola3}
\end{equation}
Notice that thanks to our choice of $K$, we can also assume that the points $w_{1,r}, w_{2,r}$ satisfy
$$\lvert \daleth(w_{j,r},r;e_\iota) -\daleth(p,r;e_\iota)\rvert\leq 2\eta\qquad\text{for }\iota=1,2.$$
This immediately shows, since $\daleth(w_j,r;e_\iota)=0$ on $K$ and $r<\delta$, that 
$$\daleth(w_{j,r},r;e_\iota)\leq (\psi(p)+2\eta)\qquad\text{for }\iota=1,2.$$
We can finally conclude by observing that this implies that for $\iota=1,2$ we have
\begin{equation}
\int_{-10r}^{10r} \int_{-10r}^{10r} \frac{\lvert g(w_{j,r}+he_\iota)-g(w_{j,r}+t e_\iota)\rvert^{q}}{\lvert h-t\rvert^{1+\alpha q}}\,dh\,dt\leq (\psi(p)+2\eta)r.
    \label{smellozzola2}
\end{equation}
Since $\varepsilon>0$ was arbitrary, we have shown that for $\Leb^2$-almost every $w\in E$, every $\eta>0$ and every $b<s$ there exists a $\delta=\delta(w,b,\eta)$ such that whenever $r<\delta$ we have point $w_{1,r}$ and $w_{2,r}$ in $E$ such that \eqref{smellozzola1}, \eqref{smellozzola3} and \eqref{smellozzola2} hold. 

However, this implies that for $\Leb^2$-almost every $w\in E$ and every $\eta$ there exists $r_i\leq \min\{\delta(w,s(1-i^{-1}),\eta),i^{-1}\}$ and points $w_{1,i}$ and $w_{2,i}$ in $E$ such that \eqref{smellozzola3} and \eqref{smellozzola2} hold and 
\begin{equation}
    \begin{split}
        &\qquad\qquad\gimel_{s(1-i^{-1})}(w_{\iota,i},r_i\,;e_\iota)\leq \eta\quad\text{for $\iota=1,2$},\\
        &\beth_{\alpha(1-i^{-1})}(w_{\iota,i},r_i\,;e_\iota)\leq \eta\quad\text{for $\iota=1,2$}\quad \text{provided }\alpha q \leq 1.
    \end{split}
\end{equation}
Finally, since $p\in E\subseteq B$, if $i$ is sufficiently big then 
\begin{equation}
\begin{split}
\gimel_b(w_{\iota,i},r_i\,;e_\iota)=&\sup_{\tau\leq r_i}\tau^{- b}\Big(\fint\lvert u(z)-u(w)\rvert^{1^*} d\Haus^1\llcorner (w+\ell_{e_\iota}\cap B(w,10\tau))(z)\Big)^\frac{1}{1^*}\\
    =&\sup_{\tau\leq r_i}\Big(\frac{\Haus(w_{\iota,i}+\ell_{e_\iota}\cap B(w_{\iota,i},10\tau)\setminus E)}{20\tau^{1+1^*b}}\Big)^\frac{1}{1^*}.
    \nonumber
\end{split}
\end{equation}
This concludes the proof.
\end{proof}

\begin{proposition}\label{frobeniustrong}
    Let $s\in (1/2,1)$ and let $g$ be a continuous $1$-form on $\R^2$ such that $\mathrm d g\in L^1_{loc}(\R^2)$. Let $B\subseteq \R^2$ be an open ball and $E$ be a Borel set such that $\mathbb{1}_E\in W^{s,1}(B)$. Then, if $g=0$ on $E$ then $\mathrm{d}g=0$ $\Leb^2$-almost everywhere on $E$.
\end{proposition}

\begin{proof}
 Thanks to Proposition \ref{prop:rettangoli}, there exists $v\in \mathbb S^{1}$, a sequence of rectangles $\mathscr{P}_i(x)$ with sides $\mathbb{L}^i_{1,3}$ and $\mathbb{L}^i_{2,4}$ parallel to $v$ and $v^\perp$ respectively. Such rectangles are contained in the balls $B(x,2r_i)$, with $r_i\leq i^{-1}$,  and the following properties hold. 
\begin{itemize}
\item[(i)] If $\mathfrak{b}_i(x)$ denotes the barycenter of $\mathscr{P}_i(x)$, then 
$|\mathfrak{b}_i(x)-x|\leq 2\varepsilon\,r_i$;
\item[(ii)] If $\ell_{1,i}$ and $\ell_{2,i}$ denote the side lengths of $\mathscr{P}_i(x)$, then for each $j=1,2$ we have $|\ell_{j,i}-2r_i|\leq 4\varepsilon\,r_i$;
\item[(iii)]
$\Haus^1\bigl(\partial\mathscr{P}_i(x)\setminus E\bigr)\leq 4\varepsilon\,r_i^{\,1+(1-i^{-1})1^*s}$, where $1^*=1/(1-s)$.  
\end{itemize}
Thanks to Lebesgue's differentiation theorem we know that for $\Leb^2$-almost every $x\in E$ we have 
\begin{equation}
    \lim_{i\to \infty}\fint_{\mathscr{P}_i(x)}\mathrm{d}g(y)dy=\mathrm{d}g(x).
    \label{lebesguepercurlPi}
\end{equation}
By Proposition \ref{divergencethm} that  
\begin{equation}
\begin{split}
     \mathrm{d}g(x)=&\lim_{i\to \infty}\fint_{\mathscr{P}_i(x)}\mathrm{d}g(y)dy\\
     =&\lim_{i\to \infty}\frac{1}{\Leb^2(\mathscr{P}_i)}
     \int\langle \mathfrak{n}\wedge g,e_1\wedge e_2\rangle\,d\Haus^{1}\llcorner \partial^*\mathscr{P}_i\\
     =&\lim_{i\to \infty}\frac{1}{\Leb^2(\mathscr{P}_i)}
     \int\langle \mathfrak{n}\wedge g,e_1\wedge e_2\rangle\,d\Haus^{1}\llcorner (\partial^*\mathscr{P}_i\setminus E),
     \nonumber
\end{split}
\end{equation}
where the last identity comes from the fact that $g=0$ on $E$. Thanks to item (iii) and the above computation we infer that 
\begin{equation}
\begin{split}
     \lvert \mathrm{d}g(x)\rvert\leq &2\lim_{i\to \infty} r_i^{-2}\Haus^1(\partial^*\mathscr{P}_i\setminus E)\Big(\fint \lvert g\rvert d\Haus^1 \llcorner (\partial^*\mathscr{P}_i\setminus E)\Big)\\
     \leq &2\lim_{i\to \infty} 4\varepsilon\,r_i^{\,-1+(1-i^{-1})1^*s}\lVert g\rVert_\infty=0,
\end{split}
\end{equation}
where the last identity above comes from the fact that $-1+s/(1-s)>0$ whenever $s>1/2$.
\end{proof}

We are ready to state the main result of this section.

\begin{proposition}\label{prop:delicata2}
Let $\alpha\in (0,1)$, $s\in (0,1/2]$, $q\in[1,\infty]$ and let $E$ be a $\Leb^2$-measurable subset of $B$ such that $\mathbb{1}_E\in W^{s,1}(B)$. Suppose $g$ is a continuous $1$-form of class $W^{\alpha,q}(B)$ such that $\mathrm{d}g\in L^1_{loc}(B)$. If $\alpha> 1-2s+q^{-1}s$, then $g=0$ for $\Leb^2$-almost every $x\in E$ implies $\mathrm{d}g=0$ for $\Leb^2$-almost every $x\in E$.
\end{proposition}

\begin{proof} Thanks to Proposition \ref{prop:rettangoli}, there exists $v\in \mathbb S^{1}$, a sequence of rectangles $\mathscr{P}_i(x)$ with sides $\mathbb{L}^i_{1,3}$ and $\mathbb{L}^i_{2,4}$ parallel to $v$ and $v^\perp$ respectively. Such rectangles are contained in the balls $B(x,2r_i)$, with $r_i\leq i^{-1}$,  and the following properties hold. 
\begin{itemize}
\item[(i)] If $\mathfrak{b}_i(x)$ denotes the barycenter of $\mathscr{P}_i(x)$, then 
$|\mathfrak{b}_i(x)-x|\leq 2\varepsilon\,r_i$;
\item[(ii)] If $\ell_{1,i}$ and $\ell_{2,i}$ denote the side lengths of $\mathscr{P}_i(x)$, then for each $j=1,2$ we have $|\ell_{j,i}-2r_i|\leq 4\varepsilon\,r_i$;
\item[(iii)]
$\Haus^1\bigl(\partial\mathscr{P}_i(x)\setminus E\bigr)\leq 4\varepsilon\,r_i^{\,1+(1-i^{-1})1^*s}$, where $1^*=1/(1-s)$.  
\item[(iv)] For every $\kappa=1,\ldots,4$, every $i\in\N$ and $\Leb^2$-almost every  there exists a constant $\psi(g,x)$ such that 
$$\int\int \frac{\lvert g(z)-g(y)\rvert^{q}}{\lvert z-y\rvert^{1+\alpha q}}d\Haus^1\llcorner \mathbb{L}_\kappa^i(z)\,d\Haus^1\llcorner \mathbb{L}_\kappa^i(y)\leq \psi(g,x)r_i,$$
\item[(v)] If 
$\alpha q\leq 1$ for every $\kappa=1,\ldots,4$ there are vertices $\mathfrak{c}_{i}^\kappa\in E$ such that
$$\Big(\fint \lvert g(z)-g(\mathfrak c_{i}^\kappa)\rvert^{q^*} d\Haus^1\llcorner \mathbb{L}^i_\kappa(z)\Big)^\frac{1}{q^*}\leq \varepsilon r_i^{(1-i^{-1})\alpha}.$$
\end{itemize}
Thanks to Lebesgue's differentiation theorem we know that for $\Leb^2$-almost every $x\in E$ we have 
\begin{equation}
    \lim_{i\to \infty}\fint_{\mathscr{P}_i(x)}\mathrm{d}g(y)dy=\mathrm{d}g(x).
    \label{lebesguepercurlPi}
\end{equation}
% Further, identifying with abuse of notations $g=\sum_j g_je_j$ with the naturally associated form $\beta(g)=\sum_j g_jdx_j$, since in $\R^2$ we have $\mathrm{curl}(g)dx_1\wedge dx_2=dg$ we have b
By Proposition \ref{divergencethm} that  
\begin{equation}
\begin{split}
     \mathrm{d}g(x)=&\lim_{i\to \infty}\fint_{\mathscr{P}_i(x)}\mathrm{d}g(y)dy\\
     =&\lim_{i\to \infty}\frac{1}{\Leb^2(\mathscr{P}_i)}
     \int\langle \mathfrak{n}\wedge g,e_1\wedge e_2\rangle\,d\Haus^{1}\llcorner \partial^*\mathscr{P}_i\\
     =&\lim_{i\to \infty}\frac{1}{\Leb^2(\mathscr{P}_i)}
     \int\langle \mathfrak{n}\wedge g,e_1\wedge e_2\rangle\,d\Haus^{1}\llcorner (\partial^*\mathscr{P}_i\setminus E),
     \nonumber
\end{split}
\end{equation}
where the last identity comes from the fact that $g=0$ on $E$.

\medskip

The proof proceeds by distinguishing three cases.

\medskip

\textsc{Case I: $\alpha q\leq  1$.} First assume that $\alpha q<1$. Thanks to item (iii) above we infer that for $\Leb^2$-almost every $w\in E$ we have 
\begin{equation}
    \begin{split}
        \lvert \mathrm{d}g(x)\rvert\leq &2\lim_{i\to \infty} r_i^{-2}\Haus^1(\partial^*\mathscr{P}_i\setminus E)\Big(\fint \lvert g\rvert d\Haus^1 \llcorner (\partial^*\mathscr{P}_i\setminus E)\Big)\\
        \overset{(iii)}{\leq} & 2 \lim_{i\to \infty} r_i^{-2}(r_i^{1+\frac{(1-i^{-1})s}{1-s}})\Big(\fint\lvert g \rvert^{q^*} d\Haus^1\llcorner (\partial^*\mathscr{P}_i\setminus E)\Big)^\frac{1}{q^*}\\
        =&2 \lim_{i\to \infty} r_i^{-2}\Big(r_i^{1+\frac{(1-i^{-1})s}{1-s}}\Big)^\frac{q^*-1}{q^*}\Big(\int\lvert g \rvert^{q^*}d\Haus^1\llcorner \partial^*\mathscr{P}_i\Big)^\frac{1}{q^*}
    \end{split}
\end{equation}
This implies in particular that, since $g(\mathfrak{c}_i^\kappa)=0$ for every $\kappa=1,\ldots,4$ we have 
\begin{equation}
    \begin{split}
        \lvert \mathrm{d}g(x)\rvert\leq&2 \lim_{i\to \infty} r_i^{-2}\Big(r_i^{1+\frac{(1-i^{-1})s}{1-s}}\Big)^\frac{q^*-1}{q^*}\Big(\sum_{\kappa=1}^4\int\lvert g(z)-g(\mathfrak{c}_i^\kappa) \rvert^{q^*}d\Haus^1\llcorner \mathbb{L}_\kappa^i\Big)^\frac{1}{q^*}\\
        \leq &8\lim_{i\to \infty} r_i^{-2}\Big(r_i^{1+\frac{(1-i^{-1})s}{1-s}}\Big)^\frac{q^*-1}{q^*}r_i^\frac{1}{q^*}\max_{\kappa=1,\ldots,4}\Big(\fint\lvert g(z)-g(\mathfrak{c}_i^\kappa) \rvert^{q^*}d\Haus^1\llcorner \mathbb{L}_\kappa^i\Big)^\frac{1}{q^*}\\
        &\qquad\qquad\overset{(v)}{\leq} 8\varepsilon r_i^{-2+\big({1+\frac{(1-i^{-1})s}{1-s}}\big)\frac{q^*-1}{q^*}+(1-i^{-1})\alpha+\frac{1}{q^*}}.
        \nonumber
    \end{split}
\end{equation}
Therefore, if the condition 
\begin{equation}
    -2+\big({1+\frac{s }{1-s}}\big)\frac{q^*-1}{q^*}+\alpha+\frac{1}{q^*}>0,
    \label{simanet1}
\end{equation}
holds, then $\mathrm{d}g(x)=0$ $\Leb^2$-almost everywhere. However one immediately sees  thanks to some algebraic computation, that \eqref{simanet1} is equivalent to $\alpha>1-2s+s/q$. 

Let us treat the case $\alpha q=1$. In this case the above computation reduces to 
\begin{equation}
\begin{split}
     \lvert \mathrm{d}g(x)\rvert\leq& 2\lim_{i\to \infty} r_i^{-1+\frac{(1-i^{-1})s}{1-s}}\lVert g\rVert_{L^\infty(\mathbb{L}_\kappa^i)}=2\lim_{i\to \infty} r_i^{-1+\frac{(1-i^{-1})s}{1-s}}\lVert g-g(\mathfrak{c}_\kappa^i)\rVert_{L^\infty(\mathbb{L}_\kappa^i)}\\
    \leq &2\lim_{i\to \infty} r_i^{-1+\frac{(1-i^{-1})s}{1-s}}r_i^{(1-i^{-1})\alpha}.
    \label{stimadgggg}
\end{split}
\end{equation}
Therefore, if 
$$\frac{-1+2s}{1-s}+\alpha>0,$$
then $\mathrm{d}g(x)=0$ for $\Leb^2$-almost everywhere. Thanks to the requirement $\alpha q=1$, we see with few algebraic computations, that \eqref{stimadgggg} is equivalent to $\alpha>1-2s+s/q$.  This concludes the proof of the first case.

\textsc{Case II: $1<\alpha q<  \infty$.} Let us notice that 
\begin{equation}
    \begin{split}
        \lvert \mathrm{d}g(x)\rvert\leq \lim_{i\to \infty} r_i^{-2}\int\lvert g\rvert d\Haus^1\llcorner \partial^* \mathscr{P}\leq \lim_{i\to \infty}r_i^{-2}\sum_{\kappa=1}^4\sum_{j\in\N} \int_{I_{j,\kappa}^i(r_i)} \lvert g\rvert d\Haus^1,
    \end{split}
\end{equation}
where with $I_{j,\kappa}^i(r_i)$ we denote the open segments of $\mathrm{int}\mathbb{L}_\kappa^\iota(r_i)$, which is the segment $\mathbb{L}_\kappa^\iota$ without its endpoints, such that 
$$\mathrm{int}\mathbb{L}_\kappa^i(r_i)\setminus \mathrm{supp}(g)=\bigcup_{j\in\N}I_{j,\kappa}^i(r_i).$$
Thanks to Proposition \ref{poincarefractional} we infer that 
\begin{equation}
    \begin{split}
         \lvert \mathrm{d}g(x)\rvert\leq &\lim_{i\to \infty}r_i^{-2}\sum_{\kappa=1}^4\sum_{j\in\N}\Leb^1(I_{j,\kappa}^i(r_i))^{1+\alpha-\frac{1}{q}}[g]_{W^{\alpha,q}(I_{j,\kappa}^i(r_i))}\\
         \leq &\lim_{i\to\infty}r_i^{-2}\sum_{\kappa=1}^4\Big(\sum_{j\in\N}\Leb^1(I_{j,\kappa}^i(r_i))^{(1+\alpha-1/q)q'}\Big)^\frac{1}{q'}\Big(\sum_{j\in\N}[g]_{W^{\alpha,q}(I_{j,\kappa}^i(r_i))}^q\Big)^\frac{1}{q}.
    \end{split}
\end{equation}
A simple computation shows that $q'(1+\alpha+1/q)\geq 1$ and hence 
\begin{equation}
    \begin{split}
        \lvert \mathrm{d}g(x)\rvert\leq\lim_{i\to \infty}r_i^{-2}\sum_{\kappa=1}^4\Big(\sum_{j\in\N}\Leb^1(I_{j,\kappa}^i(r_i))\Big)^{1+\alpha-1/q}\Big(\sum_{j\in\N}[g]_{W^{\alpha,q}(I_{j,\kappa}^i(r_i))}^q\Big)^\frac{1}{q}.
    \end{split}
\end{equation}
Thanks to item (iii), we infer that 
\begin{equation}
    \begin{split}
         \lvert \mathrm{d}g(x)\rvert&\leq\lim_{i\to \infty}r_i^{-2}\sum_{\kappa=1}^4\Big(\varepsilon r_i^{1+\frac{(1-i^{-1})s}{1-s}}\Big)^{1+\alpha-1/q}\Big(\sum_{j\in\N}[g]_{W^{\alpha,q}(I_{j,\kappa}^i(r_i))}^q\Big)^\frac{1}{q}\\
         &\leq \lim_{i\to \infty} r_i^{-2+\big(1+\frac{(1-i^{-1}s)}{1-s}\big)(1+\alpha-1/q)}\sum_{\kappa=1}^4[g]_{W^{\alpha,q}(\mathbb{L}_\kappa^i(r_i))}      \\
         &\overset{(iv)}{\leq} 4\psi(g,x)r_i^{-2+\big(1+\frac{(1-i^{-1}s)}{1-s}\big)(1+\alpha-1/q)+\frac{1}{q}}.
    \end{split}
\end{equation}
We see that if 
$$-2+\big(1+\frac{s}{1-s}\big)(1+\alpha-1/q)+\frac{1}{q}>0,$$
then $\mathrm{d}g(x)=0$ and the above inequality is easily seen to be equivalent to $\alpha >1-2s+s/q$.

\textsc{Case III: $q=  \infty$.} Since $g$ is $\alpha$-H\"older, we can estimate the sup-norm of $g$ on $\mathscr{P}_i$ as follows. By item (iii) we know that the biggest interval in which $g$ is non-zero on $\mathscr{P}_i$ has diameter $4\varepsilon\,r_i^{\,1+(1-i^{-1})s1^*}$. Hence 
\begin{equation}
    \lVert g\rVert_{L^\infty(\mathscr{P}_i)}\leq (4\varepsilon)^{\alpha}[g]_\alpha \, \,r_i^{\alpha+(1- i^{-1})s1^*\alpha}.
    \nonumber
\end{equation}
Thus, this implies in particular that 
\begin{equation}
\begin{split}
    & \qquad\qquad\lvert \mathrm{d}g(x)\rvert\leq \lim_{i\to \infty}\frac{(4\varepsilon)^{\alpha}[g]_\alpha \, \,r_i^{\alpha+(1- i^{-1})s1^*\alpha}}{\Leb^2(\mathscr{P}_i)}\Haus^{1}( \partial^*\mathscr{P}_i\setminus E)\\
     \overset{(iii)}{\leq}&\lim_{i\to\infty}\frac{(4\varepsilon)^{\alpha}[g]_\alpha \, \,r_i^{(1+(1- i^{-1})s1^*)(1+\alpha)}}{\Leb^2(\mathscr{P}_i)}\overset{(ii)}{\leq} (4\varepsilon)^{\alpha}[g]_\alpha \,\lim_{i\to \infty}r_i^{(1+(1- i^{-1})s1^*)(1+\alpha)-2}.
     \nonumber
\end{split}
\end{equation}
However it can be seen that since 
$$(1+(1- i^{-1})s1^*)(1+\alpha)-2< 0,$$
for every $i\in \N$, thanks to our choice of $\alpha$, we have 
$$\lvert \mathrm{d}g(x)\rvert\leq(4\varepsilon)^{\alpha}[g]_\alpha.$$
The arbitrariness of $\varepsilon$ concludes the proof.
\end{proof}

\section{Lusin-type results on Cantor sets with fractional boundary}\label{sectioncounterexamples}

\subsection{The statements}

This section is devoted to the proof of the variants of Lusin's theorem for gradients below, see \cite{zbMATH00021945}. The results below should be understood in the following sense. If we solve $Du=F$ for a Lipschitz datum $F$ on a set $\mathfrak{C}$ there are two quantities whose regularity necessarily trade off against each other: the regularity of the boundary of $\mathfrak{C}$ and that of $u$.

%  \begin{theorem}\label{lusinreg}
%          Let $\eta,\varepsilon>0$, $p\in[1,\infty)$, $s\in[0,p/2)$ such that $\alpha<1-2s$. Let $\Omega$ be an open bounded set in $\R^k$ and suppose that $F:\Omega\times \R^{n-k}\to \R^{k\times {n-k}}$ is a locally Lipschitz map. 
%          Then, there  are a compact set $\mathfrak{C}\subseteq \Omega$ and a function $u:\Omega\to \R^{n-k}$ such that
%          \begin{itemize}
%          \item[(i)]$\mathrm{supp}(u)\subseteq \Omega$, $\lVert u\rVert_{\infty}\leq \eta$ and $Du(x)=F(x,u(x))$ for every $x\in \mathfrak{C}$;
%              \item[(ii)] 
% $\mathscr{L}^k(\Omega\setminus \mathfrak{C})\leq \varepsilon\mathscr{L}^k(\Omega)$ and $\mathbb{1}_{\mathfrak{C}}\in W^{s,p}(\Omega)$;
% \item[(iii)] $u$ is of class $C^{1,\alpha}(\Omega)$.
%  \end{itemize}
%  Let us remark that if $s=0$, then $u$ is of class $\cap_{0<\alpha<1}C^{1,\alpha}(\Omega)$. 
%     \end{theorem}

% We can further prove a Sobolev version of the above theorem 

 \begin{theorem}\label{lusinregSobolev}
         Let $\eta,\varepsilon>0$, $q\in[1,\infty]$, $\alpha\in [0,1)$ and $0\leq s<q/(2q-1)$ such that 
         $$\alpha<1-\Big(2-\frac{1}{q}\Big)s.$$
         Let $\Omega$ be an open bounded set in $\R^k$ and suppose that $F:\Omega\times \R^{n-k}\to \R^{k\times {n-k}}$ is a locally Lipschitz map. 
         Then, there  are a compact set $\mathfrak{C}\subseteq \Omega$ and a function $u:\Omega\to \R^{n-k}$ such that
         \begin{itemize}
         \item[(i)]$\mathrm{supp}(u)\subseteq \Omega$, $\lVert u\rVert_{\infty}\leq \eta$ and $Du(x)=F(x,u(x))$ for $\Leb^k$-almost every $x\in \mathfrak{C}$;
             \item[(ii)] 
$\mathscr{L}^k(\Omega\setminus \mathfrak{C})\leq \varepsilon\mathscr{L}^k(\Omega)$ and $\mathbb{1}_{\mathfrak{C}}\in W^{s,1}(\Omega)$;
\item[(iii)] $u\in L^\infty(\Omega)\cap W^{1,q}(\Omega)$ and $Du\in W^{\alpha,q}(\Omega)$.
 \end{itemize}
In addition, if $0\leq s<1/2$ then $u$ is also of class $C^1_c(\Omega)$ and the identity 
$$Du(x)=F(x,u(x)) \qquad\text{holds everywhere on $\mathfrak{C}$}.$$
Finally, if $s=0$, then $u\in \bigcap_{0<\alpha<1}C^{1,\alpha}(\Omega)$. 
    \end{theorem}
    
In addition, we also provide the following extremal result.

\begin{theorem}\label{theoremlusin2}
      Let $\eta,\varepsilon>0$ and $d<k$. Let $\Omega$ be an open bounded set in $\R^k$ and suppose that $F:\Omega\times \R^{n-k}\to \R^{k\times {n-k}}$ is a locally Lipschitz map. 
         Then, there  are a compact set $\mathfrak{C}\subseteq \Omega$ and a function $u:\Omega\to \R^{n-k}$ such that
         \begin{itemize}
         \item[(i)]$\mathrm{supp}(u)\subseteq \Omega$, $\lVert u\rVert_{\infty}\leq \eta$ and $Du(x)=F(x,u(x))$ for every $x\in \mathfrak{C}$;
             \item[(ii)] $\mathrm{dim}_\Haus(\mathfrak{C})=d$;
\item[(iii)] $u$ is of class $C^{1,1}(\Omega)$.
 \end{itemize}
\end{theorem}

\subsection{Construction of functions with prescribed gradient.}\label{section:construction:lusin}

In this subsection we will prove a weaker version of Theorems \ref{lusinregSobolev} and \ref{theoremlusin2}. We will limit ourselves to prove that the constructed functions are $C^1_c$ independently on how the Cantor-type set $\mathfrak{C}$ is constructed. 

First of all we need to introduce some general notation that will be fixed throughout the rest of the section.

Let $\eta,\varepsilon>0$ and $\Omega$ be a bounded open set, let $\delta_0>0$, see \S\ref{cubes},
and let $\Omega'\subseteq \Omega$ be an open set containing the cubes $\Delta_0(\delta_0/2)$, see \S\ref{cubes} and for which 
$K:=\mathrm{cl}(\Omega')\subseteq \Omega$.
Suppose that $F:\Omega\times \R^{n-k}\to \R^{k\times {n-k}}$ is a locally Lipschitz map. We let
\begin{equation}
    M_1:=\lVert F\rVert_{\infty,K\times[-1,1]^{n-k}}\qquad\text{and}\qquad M_2:=\Lip(F,K\times[-1,1]^{n-k}),
    \label{eq:def:MMi}
\end{equation}
and fix 
$$\delta\leq\frac{ \min\{\eta,\mathrm{dist}(\Omega^c,\mathrm{cl}(\Omega')), \delta_0\}}{10 M_2k^\frac{3}{2}(2+12M_1\pi^2k^\frac{3}{2}+2M_1)}.$$
Finally, let $B\in \N$ and we fix an infinitesimal sequence $\boldsymbol{\rho}$ satisfying the hypothesis imposed in \eqref{ipotesirho1} and
\eqref{ipotesisurho} and let $\mathfrak{C}$ be the compact set constructed in \S\ref{defCompatto} with respect to the sequence $\boldsymbol{\rho}=(\rho_j)_{j\in\N}$. We keep the same notation of Definition \ref{definizionecubani} also for the sequence of sides of cubes $r_i$, $i\in\N\setminus\{0\}$.

\begin{proposition}\label{prop5.3}
If $\sum_{\iota=1}^\infty \rho_\iota^{-1}r_\iota^2<\infty$ then there exists a function $u$ of class $C^1_c(\Omega)$ such that 
\begin{equation}
\lVert u\rVert_\infty\leq \eta\qquad\text{and}\qquad Du(x)=F(x,u(x))\text{ on }\mathfrak{C}.
\label{lusin}
\end{equation}
On the other hand, if $\sum_{\iota\in \N} 2^{Bk\iota}r_\iota^{2q+k-1}\rho_{\iota+1}^{1-q}<\infty$ then $u\in W^{1,q}(\Omega)$, \eqref{lusin} holds $\Leb^n$-almost everywhere. 
\end{proposition}

\begin{proof}
    The construction of such $u$ is an iterative process and in order to get a consistent notation we set $u_0:=0$.

\medskip

\textsc{Base Step}
For every $Q\in\Delta_1$, see Definition \ref{definizionecubani}, let
\begin{enumerate}
    \item[(a)] $a_Q^1:=F(\mathfrak{c}(Q),0)$,
    \item[(b)] $\sigma_Q^1$ be a smooth cut off function such that $\lVert \sigma_Q^1\rVert_\infty\leq 1$, $\sigma_Q^1\equiv 1$ on $Q$, $\sigma_Q^1\equiv 0$ outside the cube with the same center as $Q$ and side length $\ell(Q)+\rho_1/2=\delta-\rho_1/2$  and such that $\lVert D\sigma_Q^1\rVert\leq 4k\rho_1^{-1}$ and $\lVert D^2\sigma_Q^1\rVert\leq 8k\rho_1^{-2}$.
\end{enumerate}
We define the map $u_1:\R^{k}\to \R^{n-k}$ as
$$u_1(x):=\sum_{Q\in\Delta_1} \sigma_Q^1(x) a_Q^1[x-\mathfrak{c}(Q)].$$
The function $u_1$ is obviously smooth and its support is contained in $\Omega$. The supremum norm of $u_1$ can be estimated as follows. Let $x\in \Omega$ and let us first note that if there does not exists a $Q\in\Delta_1$ such that $x\in\mathrm{supp} (\sigma_Q^1)$ then $u_1(x)=0$. Otherwise, since the supports of the $\sigma_Q$s are pairwise disjoint, we deduce that
\begin{equation}
    \lvert u_1(x)\rvert=\Big\lvert\sum_{Q\in\Delta_1} \sigma^1_Q(x)a_Q^1[x-\mathfrak{c}(Q)]\Big\rvert\leq  \lvert a_{Q}^1[x-\mathfrak{c}(Q)]\rvert\leq M_1\sqrt{k}\,r_1 
    \label{eq:supnormu}
\end{equation}
We now turn our attention to the estimate of the supremum norm of the gradient of $u_1$. Like in the study of the supremum norm of $u_1$ we can assume that $x\in\Omega$ is contained in $\mathrm{supp}(\sigma_Q^1)$ for some $Q\in\Delta_1$. Then
\begin{equation}
    \begin{split}
        \lvert Du_1(x)\rvert=& \lvert D\sigma^{1}_{Q}(x)\otimes a_{Q}^1[x-\mathfrak{c}(Q)]+ \sigma^{1}_{Q}(x)a_Q^1\rvert\\
        \leq& 8k\rho_1^{-1} \cdot M_1 \sqrt{k}r_1+M_1\leq (8 k^\frac{3}{2}\rho_1^{-1}r_1+1)M_1,
    \end{split}
\end{equation}
where $M_1$ is the supremum norm of $F$ introduced in \eqref{eq:def:MMi}.

Finally, since $Du_1$ coincides with $a_Q$ on $Q$ for every $Q\in\Delta_1$, we conclude that for any $x\in\cup\{Q:Q\in\Delta_1\}$ we have
\begin{equation}
    \begin{split}
        \lvert Du_1(x)-F(x,u_1(x))\rvert=&\lvert a_{Q}^1-F(x,u_1(x))\rvert=\lvert F(\mathfrak{c}(Q),0) -F(x,u_1(x))\rvert\\
        \leq& M_2(\lvert x-\mathfrak{c}(Q)\rvert+\lvert u_1(x)\rvert)\overset{\eqref{eq:supnormu}}{\leq} M_2\sqrt{k}(M_1+1)r_1.
        \nonumber
    \end{split}
\end{equation}
 Notice that $u_1$ coincides with $a_Q^1[\cdot-\mathfrak{c}(Q)]$ on each $Q$. 
This concludes the base step.

\medskip

\textsc{Inductive Step} Let us assume that we have defined inductively $u_1,\ldots,u_\iota$ satisfy the following condition. For every $\iota\in \N$ we have 
% $$u_j=a_Q^j[\cdot- \mathfrak{c}(Q)]\qquad\text{ on }Q\in \Delta_j, \text{ where }a_Q^j=F(\mathfrak{c}(Q),u_{j-1}(\mathfrak{c}(Q))),$$
% such that for every $j\in\{1,\ldots,\iota\}$ and every $k\geq j$ we have 
% $$u_k(\mathfrak{c}(Q))=u_j(\mathfrak{c}(Q))\qquad\text{for every }Q\in \Delta_j.$$
% In addition we assume that 
$$\lVert u_\iota\rVert_{L^\infty}\leq 4M_1\sqrt{k}\sum_{j=1}^{\iota} r_j.$$
Let us construct the function $u_{\iota+1}$. 
Similarly to the construction of $u_1$, for every $Q\in\Delta_{\iota+1}$ we let
\begin{enumerate}
\item[($\text{a}^\prime$)] $a_Q^{\iota+1}:=F(\mathfrak{c}(Q),u_{\iota}(\mathfrak{c}(Q)))$,
    \item[($\text{b}^\prime$)] $\sigma_Q^{\iota+1}$ be a smooth function such that $\lVert \sigma_Q^{\iota+1}\rVert_\infty\leq 1$, $\sigma_Q^{\iota+1}\equiv 1$ on $Q$, $\sigma_Q^{\iota+1}\equiv 0$ outside the cube with the same center as $Q$ and side length    
$$\ell(Q)+\rho_{\iota+1}/2=r_{\iota+1}/2-\rho_{\iota+1}/2,$$
and such that $\lVert D\sigma_Q^{\iota +1}\rVert\leq 8k\rho_{\iota+1}^{-1}$ and $\lVert D^2\sigma_Q^{\iota +1}\rVert\leq 8k\rho_{\iota+1}^{-2}$.
\end{enumerate}
We define the map $u_{\iota+1}:\R^{k}\to \R^{n-k}$ as
$$u_{\iota+1}(x):=u_\iota(x)+\sum_{Q\in\Delta_{\iota+1}}\sigma^{\iota+1}_{Q}(x)(a^{\iota+1}_{Q}-a^\iota_{\mathfrak{f}(Q)})[x-\mathfrak{c}(Q)],$$
where $\mathfrak{f}(Q)$ is the father cube of $Q$ that was introduced in Definition \ref{definizionecubani}. 

Let us check that the inductive hypothesis is satisfied. 
% Let us prove that for every $j=1,\ldots, \iota$ we have that 
% $$u_{\iota+1}(\mathfrak{c}(Q))=u_{j}(\mathfrak{c}(Q))\qquad \text{for every }Q\in \Delta_j.$$
% This is due to the fact that for every $Q\in \Delta_{\iota+1}$ we have by definition of the functions $\sigma_Q^{\iota+1}$ that 
% $$\mathfrak{T}_\iota\cap \bigcup_{Q\in \Delta_{\iota+1}}\mathrm{supp}(\sigma_Q^{\iota+1})=\emptyset,$$
% where $\mathfrak{T}_\iota:=\bigcup_{1\leq j\leq \iota}\{\mathfrak{c}(Q):Q\in \Delta_{j}\}$. Hence, for every $p\in \mathfrak{T}_\iota$ we have 
% $$u_{\iota+1}(p)=u_\iota(p).$$
% By inductive hypothesis this concludes the proof. 
% Let us now check that the inductive hypothesis on the $L^\infty$ norms of the functions $u_\iota$ is satisfied. 
Let us note that 
\begin{equation}
    \lVert a^{\iota+1}_{Q}-a^\iota_{\mathfrak{f}(Q)}\rVert\leq 2M_1,
    \label{eq:stimagrezzaa}
\end{equation}
which by definition of $u_\iota$ implies that
$$\lVert u_{\iota+1}-u_\iota\rVert_{L^\infty}\leq 4M_1 \sqrt{k}r_{\iota+1}.$$
Since by inductive hypothesis, we have $\lVert u_\iota\rVert_{L^\infty}\leq 4M_1\sqrt{k}\sum_{j=1}^{\iota-1} r_j$, the above discussion implies that
$$\lVert u_{\iota+1}\rVert_{L^\infty}\leq 4M_1\sqrt{k}\sum_{j=1}^{\iota+1} r_j,$$
which verifies the inductive hypothesis.
Notice that thanks to the above estimates, and since $r_\iota$ is summable, we know that the sequence $u_\iota$ converges in $L^\infty$ to some $u\in L^\infty$.

Let us now focus on the first order regularity of $u$. In order to do so, notice that 
$$Du_{\iota+1}-Du_\iota=\sum_{Q\in \Delta_{\iota+1}} D\sigma^{\iota+1}_{Q}(x)(a^{\iota+1}_{Q}-a^\iota_{\mathfrak{f}(Q)})[x-\mathfrak{c}(Q)]+\sigma^{\iota+1}_{Q}(x)(a^{\iota+1}_{Q}-a^\iota_{\mathfrak{f}(Q)})$$
We need to refine the estimate on $\lVert a^{\iota+1}_{Q}-a^\iota_{\mathfrak{f}(Q)}\rVert$:
\begin{equation}
    \begin{split}
         & \qquad\quad\lVert a^{\iota+1}_{Q}-a^\iota_{\mathfrak{f}(Q)}\rVert=\lVert F(\mathfrak{c}(Q),u_\iota(\mathfrak{c}(Q)))-F(\mathfrak{c}(\mathfrak{f}(Q)),u_{\iota-1}(\mathfrak{c}(\mathfrak{f}(Q))))\rVert\\
&\qquad\qquad\qquad\leq M_2\lvert\mathfrak{c}(Q)-\mathfrak{c}(\mathfrak{f}(Q))\rvert+M_2\lvert u_\iota(\mathfrak{c}(Q))-u_{\iota-1}(\mathfrak{c}(\mathfrak{f}(Q)))\rvert\\
&\qquad\qquad\qquad\leq 2M_2\sqrt{k} r_{\iota}+M_2\lvert u_\iota(\mathfrak{c}(Q))-u_{\iota-1}(\mathfrak{c}(\mathfrak{f}(Q)))\rvert.
    \end{split}
\end{equation}
Notice that on $\mathfrak{f}(Q)$ the function $u_\iota$ coincides with the linear map 
$$u_{\iota-1}(\mathfrak{c}(\mathfrak{f}(Q)))+(a_{\mathfrak{f}(Q)}^\iota-a_{\mathfrak{f}(\mathfrak{f}(Q))}^{\iota-1})[\cdot-\mathfrak{c}(\mathfrak{f}(Q))].$$ 
Thus 
\begin{equation}
    \begin{split}
        &\lVert a^{\iota+1}_{Q}-a^\iota_{\mathfrak{f}(Q)}\rVert\leq 2M_2 \sqrt{k} r_{\iota}+M_2\lVert a_{\mathfrak{f}(Q)}^\iota-a_{\mathfrak{f}(\mathfrak{f}(Q))}^{\iota-1}\rVert\lvert\mathfrak{c}(Q) -\mathfrak{c}(\mathfrak{f}(Q))\rvert\\
        &\leq 2M_2 \sqrt{k}r_{\iota}+M_2\sqrt{k}r_\iota\lVert a_{\mathfrak{f}(Q)}^\iota-a_{\mathfrak{f}(\mathfrak{f}(Q))}^{\iota-1}\rVert\overset{\eqref{eq:stimagrezzaa}}{\leq }2M_2\sqrt{k}(M_1+1)r_\iota.
        \label{stimadist}
    \end{split}
\end{equation}
Thanks to this bound and to the explicit expression for $Du_{\iota+1}-Du_\iota$ we infer the following bounds. 
Given a cube $Q\in \Delta_{\iota+1}$, if $x\in \mathrm{supp}(\sigma_Q^{\iota+1})\setminus Q$ then
\begin{equation}
    \begin{split}
        \lvert Du_{\iota+1}(x)-Du_\iota(x)\rvert&\leq 8k\rho_{\iota+1}^{-1}\cdot 2M_2\sqrt{k}(M_1+1)r_\iota \cdot \sqrt{k} r_\iota+2M_2\sqrt{k}(M_1+1)r_\iota\\
        &\leq 16k^2(M_1+1)M_2(\rho_{\iota+1}^{-1}r_\iota^2+r_\iota).
        \label{DeltaDuno}
    \end{split}
\end{equation}
and if $x\in Q$ then
\begin{equation}
    \begin{split}
        \lvert Du_{\iota+1}(x)-Du_\iota(x)\rvert&\leq 2M_2\sqrt{k}(M_1+1)r_\iota.
        \label{DeltaDdue}
    \end{split}
\end{equation}
Notice that if $\sum_{\iota=1}^\infty\rho_{\iota+1}^{-1}r_\iota^2<\infty$ then $u\in C^1_c$. 

Let us estimate the $L^q$ distance of $Du_\iota$ from $Du_{\iota+1}$ as follows
\begin{align*}
&\int \lvert Du_{\iota+1}-Du_\iota\rvert^q d\Leb^k=\sum_{Q\in \Delta_{\iota+1}}\int_{\mathrm{supp}(\sigma_Q)}\lvert Du_{\iota+1}-Du_\iota\rvert^q d\Leb^k \\
\lesssim & 2^{Bk(\iota+1)}\sup_{Q\in \Delta_{\iota+1}}\Big( \Leb^k(\mathrm{supp}(\sigma_Q)\setminus Q)(\rho_{\iota+1}^{-1}r_\iota^2+r_\iota)^q+\Leb^k(Q)r_\iota^q\Big)\\
\lesssim & 2^{Bk(\iota+1)} \big(r_{\iota+1}^{k-1}\rho_{\iota+1}(\rho_{\iota+1}^{-1}r_\iota^2+r_\iota)^q +r_{\iota+1}^kr_\iota^q\big)\lesssim_{B,q, \Omega}2^{Bk\iota} r_\iota^{2q+k-1}\rho_{\iota+1}^{1-q},
    \end{align*}
where the second last inequality comes from Jensen's inequality and the last one from fact that $2^{B k\iota}r_\iota^k\lesssim_{\Omega} 1$. By $\lesssim$ we mean that the inequalities hold true up to a constant depending on ${M_1,M_2,k,\Omega}$. This implies that, in the regime 
\begin{equation}\label{hsumm}
\sum_{\iota\in \N} 2^{Bk\iota}r_\iota^{2q+k-1}\rho_{\iota+1}^{1-q}<\infty,
\end{equation}
we have that $u\in W^{1,q}_0(\Omega)$.

Let us conclude the proof checking that \eqref{lusin} holds. Let us check this validity separately in the two different regimes. If 
$$\sum_{\iota\in \N}\rho_{\iota+1}^{-1}r_\iota^2<\infty,$$
then $u\in C^1$ and given any $x\in \mathfrak{C}$, there exists a sequence of cubes $Q_j\in \Delta_j$ such that $x\in Q_j$ for which we have 
$$Du(x)=\lim_{j\to \infty}Du_j(x)=\lim_{j\to \infty} a_{Q_j}^j=\lim_{j\to \infty} F(\mathfrak{c}(Q_j),u_j(\mathfrak{c}(Q_j)))=F(x, u(x)),$$
where the last identity comes from the continuity of $F$ and that of $u$.

In the second case, the one in which \eqref{hsumm} holds, we know that $Du_j$ converges in $L^q(\Omega)$ to $Du$. For every $j\in \N$ we have 
\begin{align*}
       \lim_{r\to 0} \fint_{B(x,r)}\lvert Du(y)-F(y,u(y))\rvert d\Leb^k\leq  & \lim_{r\to 0}\hphantom{+}\fint_{B(x,r)}\lvert Du_j(y)-Du(y)\rvert d\Leb^k \\ 
       & \hphantom{\lim_{r\to 0}} +\fint_{B(x,r)} \lvert Du_j(y)-F(y,u_j(y))\rvert d\Leb^k \\
       & \hphantom{\lim_{r\to 0}} +\fint_{B(x,r)} \lvert F(y,u(y))-F(y,u_j(y))\rvert d\Leb^k.
\end{align*}
Jensen's inequality, the continuity of $F$, the convergence in $L^\infty$ of $u_j$ to $u$ and the convergence in $L^q$ of $Du_j$ to $Du$ imply that for every $\varepsilon>0$ there exists a $j\in\N$ such that 
$$\lim_{r\to 0} \fint_{B(x,r)}\lvert Du-F(\cdot,u)\rvert d\Leb^k\leq 2\varepsilon+\lim_{r \to 0} \fint_{B(x,r)} \lvert Du_j-F(\cdot,u_j)\rvert d\Leb^k.$$
Since $x\in \mathfrak{C}$, there exists a cube $Q_j\in \Delta_j$ such that $x\in \mathrm{int}(Q)$, this implies that if $Du_j=a_{Q_j}^j$ on $B(x,r)$ provided $r$ is small enough. Hence
\begin{equation}
    \begin{split}
        \fint_{B(x,r)} \lvert Du_j(y)-&F(y,u_j(y))\rvert\, d\Leb^k(y)=\fint_{B(x,r)} \lvert a_{Q_j}^j-F(y,u_j(y))\rvert\, d\Leb^k(y)\\
        =&\fint_{B(x,r)} \lvert F(\mathfrak{c}(Q_j),u_j(\mathfrak{c}(Q_j)))-F(y,u_j(y))\rvert\, d\Leb^k(y)\\    
        \leq &\,2M_2r_j+M_2\fint_{B(x,r)}\lvert u_j(\mathfrak{c}(Q_j))-u_j(y)\rvert d\Leb^k(y).
    \end{split}
\end{equation}
However, inside the cube $Q_j$ the function $u_j$ is linear and it coincides with 
$$u_j(y)=u_j(\mathfrak{c}(Q))+a_{Q_j}^j[y-\mathfrak{c}(Q_j)],$$
since by definition $u_{j-1}(\mathfrak{c}(Q))=u_j(\mathfrak{c}(Q))$. This shows in particular that 
\begin{equation}
    \fint_{B(x,r)} \lvert Du_j(y)-F(y,u_j(y))\rvert\, d\Leb^k(y)\leq 2M_2r_j+M_2^2r_j. 
\end{equation}
This allows us to infer that 
\begin{equation}
    \begin{split}
        \lim_{r\to 0} \fint_{B(x,r)}\lvert Du-F(\cdot,u)\rvert d\Leb^k\leq 2\varepsilon+2M_2(M_2+1)r_j,
    \end{split}
\end{equation}
however the arbitrariness of $\varepsilon>0$ and of $j\in \N$ allows us to conclude that 
$$\lim_{r\to 0}\fint_{B(x,r)}\lvert Du(y)-F(y,u(y))\rvert d\Leb^k=0,$$
concluding the proof. 
\end{proof}

\subsection{Higher regularity}

This subsection is divided up into paragraphs, each of which will be devoted to the proof of Theorems \ref{lusinregSobolev} and \ref{theoremlusin2} respectively. At the beginning of each section we will specify the choices of the sequence $\boldsymbol{\rho}$ and of the dimensional constant $B$.

\subsubsection{Proof of Theorem \ref{lusinregSobolev}}\label{Caso1}
In what follows we let $q\in [1,\infty]$, $s<q/(2q-1)$ and 
$$\alpha< 1-\Big( 2 -\frac{1}{q}\Big)s.$$
We choose $B\geq 10$ and we let $\boldsymbol{\rho}=\{\rho_j\}_{j\in\N}$ be the sequence
$$\rho_j:=\Big(\frac{3\delta}{\pi^2}\Big)^\frac{1}{1-s}j^{-\frac{2}{1-s}}2^{-\frac{B}{1-s}j}.$$
Notice that the sequence $\rho_j$ is decreasing because of the choice of $B$ and that 
$$\sum_{j\in \N}2^{Bj}\rho_{j}^{1-s}= \frac{\delta}{2}.$$
Finally notice that because of the choice of $\boldsymbol{\rho}$ and the fact that $s<1/2$ we have
\begin{equation}\label{decayraggi}
   \sum_{j\in\N} \rho_{j+1}^{-1}r_j^2\lesssim_{\delta} \sum_{j\in\N} j^{\frac{2}{1-s}}2^{-\frac{1-2s}{1-s}Bj}<\infty.
\end{equation}
Thanks to Theorem \ref{prop5.3}, this implies that $u\in C^1_c(\Omega)$.

\smallskip

\textbf{Case $q<\infty$.}
It is not hard to check that if $s<q/(2q-1)$, then the series 
$$\sum_{\iota\in \N} 2^{Bk\iota}r_\iota^{2q+k-1}\rho_{\iota+1}^{1-q}\qquad\text{converges},$$
and hence $u\in W^{1,q}_0(\Omega)$. 
    Let us begin observing that for every $j\in\N$ we have 
     $$v_j:=u_{j+1}(x)-u_j(x)=\sum_{Q\in\Delta_{j+1}}\sigma^{j+1}_{Q}(x)(a^{j+1}_{Q}-a^j_{\mathfrak{f}(Q)})[x-\mathfrak{c}(Q)],$$
where $a_j$, $\mathfrak{c}(Q)$ and $\mathfrak{f}(\mathfrak{c}(Q))$ were introduced in the proof of Proposition \ref{prop5.3} and in Definition \ref{definizionecubani}.
Thus, we can write 
$$Dv_{j+1}=\sum_{Q\in \Delta_{j+1}}D\sigma_Q^{j+1}(x)\otimes (a_Q^{j+1}-a_{\mathfrak{f}(Q)}^j)[x-\mathfrak{c}(Q)]+\sum_{Q\in \Delta_{j+1}}\sigma_Q^{j+1}(x)(a_Q^{j+1}-a_Q^j).$$
Throughout the rest of the section we define 
$$\Phi_{j}:=D\sigma_Q^{j+1}\otimes (a_Q^{j+1}-a_{\mathfrak{f}(Q)}^j)[\cdot-\mathfrak{c}(Q)]\qquad\text{and}\qquad\Psi_{j}:=\sigma_Q^{j+1}(a_Q^{j+1}-a_Q^j).$$
With these notations, one immediately sees  that 
$$[Du_j]_{W^{\alpha,q}(\Omega)}\leq [Du_1]_{W^{\alpha,q}(\Omega)}+\sum_{j\in\N}[\Phi_j]_{W^{\alpha,q}(\Omega)}+\sum_{j\in\N}[\Psi_j]_{W^{\alpha,q}(\Omega)}.$$
Let us estimate $[\Phi_j]_{W^{\alpha,q}(\Omega)}$ and $[\Psi_j]_{W^{\alpha,q}(\Omega)}$ separately. 
First, we proceed with $\Psi_j$. One immediately sees  by \eqref{stimadist} that for $j$ sufficiently big we have 
\begin{equation}
    \begin{split}
\lVert \Psi_j\rVert_\infty\leq &\max_{Q\in \Delta_j}\lVert a_Q^{j+1}-a_{\mathfrak{f}(Q)}^j\rVert\leq 2M_2\sqrt{k}(M_1+1)r_j.
    \end{split}
\end{equation}
Furthermore, this implies in particular that 
$$\lVert D\Psi_j\rVert_\infty\lesssim_{M_1,M_2,k} \rho_{j+1}^{-1}r_j.$$
Hence, we have 
\begin{equation}
    \begin{split}  &\qquad\qquad\qquad\qquad[\Psi_j]_{W^{\alpha,q}(\Omega)}=\int_\Omega\int_\Omega \frac{\lvert \Psi_j(x)-\Psi_j(y)\rvert^q}{\lvert x-y\rvert^{k+\alpha q}}dx\,dy\\
        &=\int_\Omega\int_{B(y,\rho_j)} \frac{\lvert \Psi_j(x)-\Psi_j(y)\rvert^q}{\lvert x-y\rvert^{k+\alpha q}}dx\,dy+\int_\Omega\int_{B(y,\rho_j)^c} \frac{\lvert \Psi_j(x)-\Psi_j(y)\rvert^q}{\lvert x-y\rvert^{k+\alpha q}}dx\,dy\\
       & \qquad\lesssim_{M_1,M_2,k,\Omega} \rho_{j+1}^{-q}r_j^q\int_0^{\rho_{j+1}} s^{q(1-\alpha)-1}ds+r_j^q\int_{\rho_{j+1}}^\infty s^{-1-\alpha q}ds\lesssim_{\alpha,q} \rho_{j+1}^{-\alpha q}r_j^q.
    \end{split}
\end{equation}
This shows because of the choice of the sequence $\rho_j$ and \eqref{decayraggi} that if $\alpha<1-s$ then 
$$\sum_{j\in \N} [\Psi_j]_{W^{\alpha,q}(\Omega)}<\infty.$$
Let us focus on the more delicate estimate of $[\Phi]_{W^{\alpha,q}(\Omega)}$. Let us notice that 
\begin{equation}
    \lVert \Phi_j\rVert_\infty\lesssim_k\rho_{j+1}^{-1}\max_{Q\in \Delta_j}\lVert a_Q^{j+1}-a_{\mathfrak{f}(Q)}^j\rVert\mathrm{diam} Q\lesssim_{M_1,M_2,k} \rho_{j+1}^{-1} r_j^2.
\end{equation}
On the other hand
\begin{equation}
    \begin{split}
\lVert D\Phi_j\rVert_\infty\leq&\Big(\lVert D\sigma_Q^{j+1}\rVert_\infty+\lVert D^2\sigma_Q^{j+1}\rVert_\infty\mathrm{diam}(Q)\Big)\max_{Q\in \Delta_j}\lVert a_Q^{j+1}-a_{\mathfrak{f}(Q)}^j\rVert\\
&\qquad\qquad\lesssim_{M_1,M_2,k} \rho_j^{-1} r_j+\rho_j^{-2}r_j^2\lesssim\rho_j^{-2}r_j^2,
\nonumber
    \end{split}
\end{equation}
where the last inequality comes from the choice of $\rho_j$ and the fact that $\rho_{j}^{-1}r_j\geq 1$. This implies in particular that 
\begin{equation}
    \begin{split}
&\qquad\qquad\qquad\qquad[\Phi_j]_{W^{\alpha,q}(\Omega)}=\int_\Omega\int_\Omega \frac{\lvert \Phi_j(x)-\Phi_j(y)\rvert^q}{\lvert x-y\rvert^{k+\alpha q}}dx\,dy\\
        &=\int_\Omega\int_{B(y,\rho_j)} \frac{\lvert \Phi_j(x)-\Phi_j(y)\rvert^q}{\lvert x-y\rvert^{k+\alpha q}}dx\,dy+\int_\Omega\int_{B(y,\rho_j)^c} \frac{\lvert \Phi_j(x)-\Phi_j(y)\rvert^q}{\lvert x-y\rvert^{k+\alpha q}}dx\,dy\\
        &=\int_{\mathfrak{C}_j\setminus \mathfrak{C}_{j+1}}\int_{B(y,\rho_j)} \frac{\lvert \Phi_j(x)-\Phi_j(y)\rvert^q}{\lvert x-y\rvert^{k+\alpha q}}dx\,dy+\int_{\mathfrak{C}_j\setminus \mathfrak{C}_{j+1}}\int_{B(y,\rho_j)^c} \frac{\lvert \Phi_j(x)-\Phi_j(y)\rvert^q}{\lvert x-y\rvert^{k+\alpha q}}dx\,dy\\
        &\leq \Leb^{k}(\mathfrak{C}_j\setminus \mathfrak{C}_{j+1})\Big( \int_{B(y,\rho_j)}\frac{\lVert D\Phi_j\rVert_\infty^q}{\lvert x-y\rvert^{k+(\alpha-1) q}}\,dx\,dy+\int_{B(y,\rho_j)^c}\frac{2\lVert \Phi_j\rVert_\infty^q}{\lvert x-y\rvert^{k+\alpha q}}\,dx\,dy\Big)\\
         &\lesssim_{M_1,M_2,k,\Omega} \Leb^{k}(\mathfrak{C}_j\setminus \mathfrak{C}_{j+1})\bigg(\rho_{j+1}^{-2q}r_j^{2q}\int_0^{\rho_{j+1}} s^{q(1-\alpha)-1}ds+\rho_{j+1}^{-q}r_j^{2q}\int_{\rho_{j+1}}^\infty s^{-1-\alpha q}ds\bigg)\\
         &\!\!\overset{\eqref{stimamassadiffCC}}{\lesssim} j^{-\frac{2Bj}{1-s}}2^{-\frac{sBj}{1-s}}\rho_{j+1}^{-  q(1+\alpha)}r_j^{2q}\lesssim j^\frac{-2+q(1+\alpha)}{1-s}2^{\frac{-s+q(1+\alpha)-2q(1-s)}{1-s}Bj}.
         \nonumber
    \end{split}
\end{equation}
Therefore, if 
$$\alpha<1-\Big(2-\frac{1}{q}\Big)s,$$
we see that 
$$\sum_{j\in\N}[\Phi_j]_{W^{\alpha,q}(\Omega)}<\infty.$$
This concludes the proof of the fact that $u\in W^{1+\alpha,q}(\Omega)\cap L^\infty(\Omega)$ in the sense that $u\in W^{s,1}(\Omega)\cap L^\infty(\Omega)$ and $Du\in W^{\alpha,q}(\Omega)$.

\smallskip

\textbf{Case $q=\infty$.} 
In this specific case we know that $-1+\alpha+2s<0$ and hence $s<1/2$. As remarked above in this regime the function $u$ is automatically $C^1_c$ thanks to our choice of the sequence $\boldsymbol{\rho}$. 
In this paragraph we study H\"older estimates for $Du$. 
Let $\alpha\in(0,1]$ and note that for every $x,y\in \Omega$ and any $\iota\in\N$ we have
\begin{equation}
    \begin{split}
        \frac{\lvert Du(x)-Du(y)\rvert}{\lvert x-y\rvert^\alpha}\leq& \frac{2\lVert Du-Du_\iota\rVert_\infty}{\lvert x-y\rvert^\alpha}+\frac{\lvert Du_\iota(x)-Du_\iota(y)\rvert}{\lvert x-y\rvert^\alpha}\\
        \leq& \frac{2\lVert Du-Du_\iota\rVert_\infty}{\lvert x-y\rvert^\alpha}+\mathrm{Lip}(Du_\iota)\lvert x-y\rvert^{1-\alpha}.
        \label{stimaholder0}
    \end{split}
\end{equation}
For every $Q\in \Delta_1$ and every 
$x,y\in Q$, since both in \S\ref{Caso1} and in \S\ref{Caso2} the sequence $\rho_{\iota+1}^{-1}r_\iota^2$ is decreasing and $\ell(Q)=r_1<\rho_1^{-1}r_0^2$,
there exists a $\iota\in\N$, depending on $x,y$, such that 
\begin{equation}
\rho_{\iota+1}^{-1}r_\iota^2\leq \lvert x-y\rvert^\alpha\leq \rho_{\iota}^{-1}r_{\iota-1}^2.
\label{eq:num:num11}
\end{equation}
In order to estimate $\mathrm{Lip}(Du_\iota)$, where now $\iota$ is the one for which \eqref{eq:num:num11} holds, we equivalently bound $\lVert D^2 u_\iota\rVert_\infty$. One immediately sees  that if $x\not\in \cup \Delta_\iota$, then $D^2u_\iota(x)=D^2u_{\iota-1}(x)$. On the other hand, if $x\in Q$ for some $Q\in \Delta_\iota$ we have $D^2u_{\iota-1}(x)=0$ and thus
\begin{equation}
\begin{split}
    \lvert D^2u_\iota(x)\rvert&= \lvert D^2\sigma^{\iota}_{Q}(x)\otimes(a^{\iota}_{Q}-a^{\iota-1}_{\mathfrak{f}(Q)})[x-\mathfrak{c}(Q)]+ D\sigma^{\iota}_{Q}(x)(a^{\iota}_{Q}-a^{\iota-1}_{\mathfrak{f}(Q)})\rvert\\
    &\leq (8k^\frac{3}{2} \rho_\iota^{-2}r_\iota+8k \rho_\iota^{-1})\lVert a^{\iota}_{Q}-a^{\iota-1}_{\mathfrak{f}(Q)}\rVert\\
    &\leq  4M_1 \sqrt{k}(8k^\frac{3}{2} \rho_\iota^{-2}r_\iota+8k \rho_\iota^{-1}) r_{\iota}.  
      \label{eq:num:num12}
      \end{split}
\end{equation}
The first identity above comes from the fact that $u_{\iota-1}$ is linear on each $Q\in \Delta_\iota$.
% Thanks to \eqref{id:B3uff} and the summability  of the sequence $(B3)_\iota$, we have 
% $$\mathrm{Lip}(Du_\iota)\lesssim_{\lVert u\rVert_{C^1}}\max_{j\leq \iota}(8k^\frac{3}{2} \rho_j^{-2}r_j+8k \rho_j^{-1})r_j.$$
% On the other hand, we have
% \begin{equation}
%     \begin{split}
%         \lVert Du-Du_\iota\rVert_\infty\leq& \sum_{\tau=\iota}^\infty \lVert Du_{\tau+1}-D_{\tau}u\rVert_\infty\leq \sum_{\tau=\iota}^\infty(B3)_{\tau+1}.
%     \end{split}
% \end{equation}
Therefore, \eqref{stimaholder0} implies that for every $x,y\in Q$ we have
\begin{equation}
\begin{split}
    &\hphantom{\le}\frac{\lvert Du(x)-Du(y)\rvert}{\lvert x-y\rvert^\alpha}\leq 2\lvert x-y\rvert^{-\alpha}\sum_{\tau=\iota}^\infty\lVert Du_{\tau+1}-Du_{\tau}\rVert_{L^\infty}+\mathrm{Lip}(Du_\iota)\lvert x-y\rvert^{1-\alpha}\\
&\lesssim_{M_1,M_2,k}(\rho_{\iota+1}^{-1}r_\iota^2)^{-1}\sum_{\tau=\iota}^\infty(\rho_{\tau+1}^{-1}r_\tau^2+r_\tau)+\big(\max_{j\leq \iota}( \rho_j^{-2}r_j+ \rho_j^{-1})r_j\big)(\rho_{\iota}^{-1}r_{\iota-1}^2)^\frac{1-\alpha}{\alpha},
\label{stimagradientsia}
\end{split}
\end{equation}
where the last inequality follows from \eqref{DeltaDuno}, \eqref{DeltaDdue} and \eqref{eq:num:num11}. Since $\rho_{\iota+1}^{-1}r_\iota$ is increasing, we infer that 
\begin{equation}
    \begin{split}
        &\frac{\lvert Du(x)-Du(y)\rvert}{\lvert x-y\rvert^\alpha}\lesssim\frac{2^{\frac{1-2s}{1-s}B\iota}}{(\iota+1)^{\frac{2}{1-s}}}\big( 2^{-B\iota}+ \sum_{\tau=\iota}^\infty\rho_{\tau+1}^{-1}r_\tau^2\big)+\iota^\frac{\alpha+1}{\alpha}2^{\frac{-1+\alpha+2s}{\alpha(1-s)}B\iota},
        \label{stima1dup}
    \end{split}
\end{equation}
where the implicit constant depends only on $\delta, k, M_1, M_2$. 
Observe that, because of the choices of $\alpha$ and $s$, the function $\iota \mapsto \iota^\frac{\alpha+1}{\alpha}2^{\frac{-1+\alpha+2s}{\alpha(1-s)}B\iota}$ is bounded. In order to conclude that $Du$ is $\alpha$-H\"older we first need to estimate for every $\iota\in \N$ the series $\sum_{\tau=\iota}^\infty \rho_{\tau+1}^{-1}r_\tau^2$. 
By definition of $\rho_{\iota+1}$ and $r_\iota$ we infer that 
\begin{equation}
    \sum_{\tau=\iota}^\infty \rho_{\tau+1}^{-1}r_\tau^2\lesssim_{\delta,s} \sum_{\tau=\iota}^\infty (\tau+1)^{\frac{2}{1-s}} 2^{-\frac{1-2s}{1-s}B\tau}\lesssim_{s,B}\Gamma\Big(\frac{2}{1-s}+1,\log 2\,\iota\frac{1-2s}{1-s}B\Big),
\end{equation}
    where $\Gamma(x,s)$ here denotes the incomplete $\Gamma$ function. Notice that thanks to the properties of the incomplete $\Gamma$ function we have 
    $$\lim_{\iota \to \infty} \frac{\Gamma(s,x)}{x^{s-1}e^{-x}}=1,$$
hence, for $\iota$ big enough, thanks to few algebraic computations, we have that
\begin{equation}
    \sum_{\iota\geq \iota_0}\rho_{\tau+1}^{-1}r_\tau^2\lesssim_{s,\delta,B}  \iota^\frac{2}{1-s} 2^{-\frac{1-2s}{1-s}B\iota}. 
    \label{stimafiga}
\end{equation}
Finally, we can estimate the $[Du]_\alpha$ seminorm  thanks to \eqref{stimafiga} and \eqref{stima1dup} as follows
\begin{equation}
    \begin{split}
[Du]_\alpha\lesssim_{M_1,M_2,k,\delta,s,\delta,B}(\iota+1)^{-\frac{2}{1-s}}2^{-\frac{s}{1-s}B\iota}+2 +\iota^\frac{\alpha+1}{\alpha}2^{\frac{-1+\alpha+2s}{\alpha(1-s)}B\iota}.
    \end{split}
\end{equation}
The function on the right-hand side is bounded in $\iota$ and hence $[Du]_\alpha<\infty$ and hence $u$ is of class $C^{1,\alpha}$. Finally, notice that if $s=0$ then $[Du]_\alpha$ is finite for every $\alpha>0$ and hence $u\in \bigcap_{0<\alpha<1} C^{1,\alpha}(\Omega)$.
This exhausts the last case and concludes the proof.

\subsubsection{Proof of Theorem \ref{theoremlusin2}}
\label{Caso2}
Let $d<k$.
Further we let $\lambda\in(0,1)$ be such that 
$\lambda:=2^{-\frac{B(k-d)}{d}}$,
and let us choose $\boldsymbol{\rho}=\{\rho_\iota\}_{\iota\in\N}$ to be the sequence
\begin{equation}
    \rho_\iota:=\delta (1-\lambda)\lambda^\iota2^{-B\iota}.
    \label{eq:rhoiCzero}
\end{equation}
Notice that the sequence $\rho_\iota$ is decreasing and that 
$\sum_{\iota\in \N}2^{B\iota}\rho_{\iota}=\delta$. The definition of $\rho_\iota$ implies in particular that 
\begin{equation}
    r_\iota= \delta \lambda^{\iota+1}2^{-B\iota}.
    \label{eq:r.iotaCzero}
\end{equation}
% Further, let us notice for this choice of $\boldsymbol{\rho}$ implies that $\mathfrak{s}_\iota=8k^\frac{3}{2}(1-\lambda)^{-1}\delta \lambda^{\iota-1}2^{-B(\iota-1)}$ and hence 
% \begin{equation}
% \begin{split}
%      \mathfrak{N}=\sum_{\iota\geq 2} (8 k^{\frac{3}{2}} \rho_{\iota+1}^{-1}r_{\iota+1}+&1)r_\iota=8k^\frac{3}{2}(1-\lambda)^{-1}\delta\frac{\lambda^22^{-2B}}{1-\lambda 2^{-B}}.
%     \label{stmaNNN}
% \end{split}
% \end{equation}
Arguing verbatim as in the proof of Proposition \ref{Caso1}, see \eqref{stimagradientsia}, with the choice $\alpha=1$, $q=\infty$ and $s=0$, we see that in this case we have 
\begin{equation}
    \begin{split}
        \sup_{x,y\in \R^k} \frac{\lvert Du(x)-Du(y)\rvert}{\lvert x-y\rvert}\lesssim_{M_1,M_2,k}\sup_{\iota \in \N}\frac{\sum_{\tau=\iota}^\infty(\rho_{\tau+1}^{-1}r_\tau^2+r_\tau)}{\rho_{\iota+1}^{-1}r_\iota^2}+\max_{j\leq \iota} \rho_j^{-2}r_j^2+ \rho_j^{-1}r_j\lesssim_{\lambda,\delta}1,
        \nonumber
    \end{split}
\end{equation}
where the last bound is an immediate consequence of the choice of the sequence $\rho_\iota$ and few omitted algebraic computations.

As in the previous step we let $\mathfrak{D}:=\mathfrak{C}(\delta, \boldsymbol{\rho},B,\Omega)$ the compact set constructed in \S \ref{defCompatto}. Let us notice that thanks to Proposition \ref{dimensione}, we have that $\mathrm{dim}_\Haus(\mathfrak{C})=d$.

\section{Frobenius theorems}

\subsection{From tangency sets to a PDE constraint}\label{totangencytoDPE}

Let $V$ be a $k$-dimensional distribution  on $\R^n$ spanned by the system of orthonormal $C^1$ vector fields $\{X_1,\ldots,X_k\}$ and $S$ a $k$-dimensional submanifold of $\R^n$. Without loss of generality we can assume that $0\in\tau(S,V)$ and
$$\Tan(S,0)=\Span(e_1,\ldots, e_k)=:W.$$
Thanks to the regularity of $V$ and the fact that $V(0)=W$, there exists an $\newr\label{r:1}>0$ for which $V(x)$ is the graph of a linear function $M(x):W\to W^\perp$ whenever $x\in U(0,\oldr{r:1})$.

Moreover since $S$ is a $k$-dimensional embedded surface of class $C^1$ there are an $0<\newr\label{r:2}<\oldr{r:1}$, an open neighbourhood $U$ of $0$ in $W$ and a function $f:U\to W^\perp$ of class $C^1$ such that:
$$\gr(f)=S\cap U(0,\oldr{r:2}).$$
Since $\Tan(S,f(y))=\text{im}[Df(y)]$ for any $y\in U$, we can express the tangency set $\tau(S,V)$ in terms of $f$ and $M$:
\begin{equation}
\tau(S,V)\cap U(0,\oldr{r:2})=f\big(\{y\in U: Df(y)=M(f(y))\}\big).
\label{eq:2}
\end{equation}

The following proposition links the non-involutivity of $V$ to the curl of the matrix field $M$. This is a rephrasing of the standard connection between Frobenius Theorem and Poincar\'e Lemma, but we include a proof here for the sake of consistency with our notation and with our setting of the problem.

\begin{proposition}\label{Non-inv}
Suppose $V$ is non-involutive at $0$. Then there are a radius $0<\newr\label{r:3}<\oldr{r:2}$ and indices $a,b\in\{1,\ldots,k\}$ and $p\in\{1,\ldots,n-k\}$, such that:
$$\partial_a M_{p,b} -\partial_b M_{p,a}\neq0 \text{ on }U(0,\oldr{r:3}).$$
\end{proposition}

\begin{proof}
For any $i\in\{1,\ldots,k\}$ we define the vector fields $X_i:U(0,\oldr{r:2})\to\R^n$ by
\begin{equation}
X_i(z):=e_i+M(z)[e_i].
\label{eq:1027}
\end{equation}
The vector fields $X_i$ are of class $C^1$ and for any $z\in U(0,\oldr{r:2})$ the vectors $\{X_1(z),\ldots,X_k(z)\}$ span $V(z)$. Since $V$ is not involutive at $0$, we can find $a,b\in\{1,\ldots,k\}$ such that $[X_{a},X_{b}](0)\neq 0$.
Indeed, if this was not the case, we would have:
$$\bigg[\sum_{p=1}^k\alpha_p X_p,\sum_{q=1}^k\beta_q X_q\bigg](0)=\sum_{q=1}^k\bigg(\sum_{p=1}^k \alpha_p(0)\bd_p\beta_q(0)-\beta_p(0)\partial_p\alpha_q(0)\bigg)e_q\in W,$$
for any $\alpha_p,\beta_q\in C^1( U(0,\oldr{r:2}))$. This would be in contradiction with the fact that $V$ is not involutive at $0$. Thanks to the definition of the vector fields $X_i$ in \eqref{eq:1027} together with few computations that we omit, for any $i,j\in\{1,\ldots,k\}$ we have:
\begin{align}\label{eq:1}
    [X_i,X_j]= & \hphantom{+}\ \sum_{p=1}^k (\partial_i M_{p,j} -\partial_j M_{p,i})e_p
    \\ & +\!\!\!\sum_{p=k+1}^n\sum_{q=k+1}^n  (M_{q-k,i}\partial_p M_{p-k,j}-M_{q-k,j}\partial_p M_{p-k,i})e_p.
\end{align}
Therefore, identity \eqref{eq:1} together with the fact that $M(0)=0$ implies that:
    \begin{equation}
       0\neq[X_a,X_b](0)=\sum_{p=1}^k (\partial_a M_{p,b}(0) -\partial_b M_{p,a}(0))e_p.
        \label{eq:1028}
    \end{equation}
From \eqref{eq:1028} we deduce in particular that there is a  $p\in\{1,\ldots,n-k\}$ such that:
$$\partial_a M_{p,b}(0) -\partial_b M_{p,a}(0)\neq 0.$$
The existence of $\oldr{r:3}$ follows by the regularity of $M$.
\end{proof}

The previous proposition has the following consequence:

\begin{proposition}
Suppose that $f$ is  of class $C^{1,1}$ and $V$ is non-involutive at any point of $U(0,\oldr{r:3})$. Then:
$$\Haus^k(\tau(S,V)\cap U(0,\oldr{r:3}))=0.$$
\end{proposition}

\begin{proof}
Thanks to identity \eqref{eq:2}, we just need to prove that the set
$\mathcal{T}:=\{y\in U: Df(y)=M(f(y))\}$ is Lebesgue-null.
Thanks to Whitney's extension theorem, see  in \cite[Theorem 3.1.15]{Federer1996GeometricTheory}, for any $\varepsilon>0$ there exists a function $g:U\to W^\perp$ of class $C^2$ such that, defined $K:=\{y\in U: f(x)=g(x)\}$, we have: $\Leb^{k}(U\setminus K)<\varepsilon$.
Moreover, since $f$ is of class $C^{1,1}$, we also deduce that, at $\Leb^k$-almost every $x\in K$:
$$D f(x)=Dg(x)\qquad \text{and} \qquad D^2 f(x)=D^2 g(x).$$
Proposition \ref{Non-inv} implies that for some $i,j\in\{1,\ldots,k\}$ and $p\in\{1,\ldots,n-k\}$ we have:
\begin{equation}
    \begin{split}
      0\neq\partial_i M_{p,j}(x) -\partial_j M_{p,i}(x)=\bd^2_{i,j}f_p(x)-\bd^2_{j,i}f_p(x)=\bd^2_{i,j}g_p(x)-\bd^2_{j,i}g_p(x),
    \end{split}
    \label{eq:3}
\end{equation}
for $\Leb^k$-almost every $x\in K$. Since $g$ is of class $C^2$, Schwarz Theorem together with \eqref{eq:3} implies that $\Leb^k(K\cap\mathcal{T})=0$. Therefore by arbitrariness of $\varepsilon$ the conclusion follows. 
\end{proof}

An immediate consequence of the above proposition, is the following:

\begin{corollary}\label{cor:Bal}
If $S$ is a $k$-dimensional surface of class $C^{1,1}$ and $V$ is non-involutive at any point of $\R^n$, then 
$\Haus^k(\tau(S,V))=0$.
\end{corollary}

\subsection{Frobenius theorem and fine structure of tangency set}

As we have seen in the previous subsection, non-involutivity can be characterized by a PDE constraint. This observations allows us to bridge 
the geometric structure of the tangency set with the possibility of obtaining a Frobenius-type theorem. First, we need to introduce some notations.

\begin{definition}\label{Yreg}
Let $\alpha \in (0,1)$ and $q \in [1, \infty]$.  
We say that a closed set $S \subset \mathbb{R}^n$ is a $k$-dimensional submanifold of class $Y^{1+\alpha, q}$ if
\begin{itemize}
    \item[(i)] $S$ is an embedded $k$-dimensional submanifold of class $C^1$, and
    \item[(ii)] there exists an atlas $\mathscr{A}$ of $C^1$-regular maps $\varphi_j:U_j\subseteq V_j\to V_j^\perp$ where $V_j\in\Gr(k,n)$, $U_j\subseteq V_j$ is relatively open in $V_j$, $\varphi_j(U_j)\subseteq S$ and satisfies $$D\varphi_j \in W^{\alpha, q}(U_j,\R^{k\times (n-k)}).$$
\end{itemize}
Notice that trivially, if $q=\infty$, then $S$ is of class $C^{1,\alpha}$.
\end{definition}

For sets with low-regularity boundary we obtain the following result.

\begin{theorem}\label{main:incorpus}
Let $V$ be a $k$-dimensional distribution in $\mathbb{R}^n$ of class $C^1$, and let $q \in [1, \infty]$, $s, \alpha \in (0,1)$ be such that $s\in (0,1/2]$ and
\[
\alpha > 1 - \left(2 - \frac{1}{q} \right) s.
\]
Suppose $S$ is a $k$-dimensional $Y^{1+\alpha, q}$-rectifiable set, see Definitions \ref{def:rectifiable} and \ref{Yreg}, and let 
$$E \subseteq \tau(S, V)\cap N(V),$$ 
be a Borel set, where $\tau(S, V)$ denotes the tangency set defined in \S\ref{def:tan} and $N(V)$ denotes the non-involutivity set of the distribution $V$, see \S\ref{def:inv}. If the characteristic function $\mathbb{1}_E$ belongs to $W^{s, 1}(S)$, where  the space $W^{s,1}(S)$ was introduced in  \S\ref{parag:def:sobolev:surface}, then $\mathcal{H}^k(E) = 0$.
\end{theorem}

\begin{proof}
Thanks to the definition $Y^{1+\alpha,q}$-rectifiability and that of $W^{s,p}(S)$, without loss of generality we can assume that $S$ is an embedded $Y^{1+\alpha,q}$-regular $k$-dimesional submanifold. Take a chart $\varphi_j:U_j\subseteq V_j\to V_j^\perp$ in the atlas $\mathscr{A}$ yielded by our assumption that $S$ is of class $Y^{1+\alpha,q}$ and suppose $E\cap \varphi_j(U_j)\neq \emptyset$. 
Denote with $\Gamma_j$ the graph of $\varphi_j$ over $U_j$ and note that
$$[u]_{W^{s,1}(\varphi_j(U_j))}\leq [u]_{W^{s,1}(S)}<\infty.$$
By Proposition \ref{equivalenza}, we finally infer that, defining $\widetilde{E}:=\pi_{V_j}(E\cap \varphi_j(U_j))$, we have 
$\mathbb{1}_{\widetilde{E}}=\mathbb{1}_E\circ \varphi_j\in W^{s,1}(U_j)$. Note that $\widetilde{E}$ is a Suslin set, since $E\cap \varphi_j(U_j)$ is Borel.

By contradiction we assume that $\Haus^k(E\cap \varphi_j(U_j))>0$ for some $j$. In what follows we will drop the dependence of all the objects above from $j$ and it is understood that all these computations can be made in any chart of the atlas.

Arguing as at the beginning of Subsection \ref{totangencytoDPE}, without loss of generality we can assume that
\begin{itemize}
\item[(i)]   $0\in U$, that $\varphi(0)=0$, that $D\varphi(0)=0$ and that $0$ is an $\Haus^k$-density point for $\tau(S,V)$ in $S$;
    \item[(ii)] in a neighborhood of $0$ the planes of the distribution $V$ coincide with the graphs of the $C^1$ matrix field $M:\R^k\to \R^{k\times n-k}$. 
    \item[(iii)] in a neighbourhood of $0$ in $\R^k$ we have that $(w,\varphi_(w))\in \tau(S,V)$ if and only if $D\varphi(w)=M(w)$.
\end{itemize}
Thanks to the above reduction we see that $\widetilde{E}$ has positive measure $\Haus^k(\widetilde{E})>0$ and that $0$ is a density point for $\widetilde{E}$ in $V$. 

Let us note that by Proposition \ref{Non-inv}, for every system $\mathscr{E}:=\{e_1,\ldots,e_k\}$ of orthonormal coordinates of $\R^k$, we know that there exists $p=1,\ldots,n-k$, and $a,b\in \{1,\ldots,k\}$ such that 
\begin{equation}
    \partial_a M_{p,b} -\partial_b M_{p,a}\neq0,
    \label{curl1}
\end{equation}
in an open neighborhood $U\subseteq \R^k$ of $0$. 

Let us define $\mathfrak{m}:=\mathfrak{m}_{\mathscr{E},a,b}:\R^2\to \R^2$ be the vector field 
$$\mathfrak{m}(z):=(M_{p,a}(z),M_{p,b}(z)).$$

Clearly, \eqref{curl1} implies that $d\mathfrak{m}\neq 0$ in a neighbourhood of $0$.
Let us now introduce some notation. In $\R^k$, we denote by $\Pi_{a,b}$ the plane 
$$\Pi_{a,b}:=\mathrm{span}(\{e_1,\ldots, e_{a-1},e_{a+1},\ldots,e_{b-1},e_{b+1},\ldots,e_k\}).$$
Let $B$ be a small ball contained in $\Omega$ centered at $0$ and let $\eta$ be a smooth and positive cutoff function whose support is contained in $B$ and such that $\eta=1$ on $\frac{9}{10}B$. 
By Theorem \ref{restrictionW1,p} and the arbitrariness of the choice of the system of coordinates $\mathscr{E}$, we know that for $\Haus^{k-2}$-almost every $w\in \Pi_{a,b}$ we have, defined $u:=\eta \mathbb{1}_E$, that $u\lvert_{ w+W_{a,b}}\in W^{s,1}(\R^2)$ and $D\varphi\lvert_{ w+W_{a,b}}\in W^{s,1}(\R^2,\R^{n-k})$,
where $W_{a,b}:=\mathrm{span}(e_a,e_b)$. In particular, by Fubini for $\Haus^{k-2}$-almost every $w\in \Pi_{a,b}$ we have 
$$\Haus^k(\tilde{E}\cap w+W_{a,b})\geq \lVert u\rVert_{L^p(w+W_{a,b})}^p>0.$$

Let $g$ be the form $2$-form of class $W^{\alpha,q}$ that coincides with the differential of the coordinate function $\varphi_p$ on the plane $w+W_{a,b}$ or in other words
$$g(z):=\partial_a \varphi_p \,dx_a+\partial_b \varphi_p \, dx_b,\qquad\text{with }z\in w+W_{a,b}.$$
Thanks to item (iii) above, let us note that we have $\mathfrak{m}-g=0$ on $\tilde{E}\cap w+W_{a,b}$.
It is not hard to check that $\mathrm{d}g=0$ as a distribution and hence $\mathrm{d}g\in L^1(\R^2)$.

Let us conclude the proof of the proposition. 
Let us note that $\mathrm{d}(\mathfrak{m}-g)$ is in $L^1_{\rm loc}(\R^2)$ and it coincides with $\mathrm{d}\mathfrak{m}$. Indeed, for every smooth $2$-current $T$, by definition of distributional differential we have 
$$\langle T,\mathrm{d}(\mathfrak{m}-g)\rangle=\langle \partial T,\mathfrak{m}-g\rangle=\langle \partial T,\mathfrak{m}\rangle-\langle \partial T,g\rangle=\langle \partial T,\mathfrak{m}\rangle.$$
Now, since $\mathfrak{m}$ is of class $C^1$, its distributional differential coincides with the classical one and it can be therefore represented by a continuous function. 

This implies by Proposition \ref{prop:delicata2} that 
$$0=\mathrm{d}(\mathfrak{m}-g)=\mathrm{d}\mathfrak{m}\qquad\text{on }\tilde{E}\cap w+W_{a,b}.$$
This is however in contradiction with the fact that $\mathrm{d}\mathfrak{m}\neq 0$. 
\end{proof}

For tangency sets with fractional boundary of higher regularity  we obtain a Frobenius-type theorem for standard rectifiable sets.

\begin{theorem}\label{main:incorpusv2}
    Let $V$ be a $k$-dimensional distribution in $\R^n$ of class $C^1$ and suppose $S$ is a $k$-dimensional rectifiable set. Let 
    $$E\subseteq \tau(S,V)\cap N(V),$$
    be a Borel set, where $\tau(S, V)$ denotes the tangency set defined in \S\ref{def:tan} and $N(V)$ denotes the non-involutivity set of the distribution $V$, see \S\ref{def:inv}. 

    If the characteristic function $\mathbb{1}_E$ belongs to $W^{s, 1}(S)$ for some $s>1/2$, where  the space $W^{s,1}(S)$ was introduced in  \S\ref{parag:def:sobolev:surface}, then $\mathcal{H}^k(E) = 0$.
\end{theorem}

\begin{proof}
    The proof follows that of Theorem \ref{main:incorpusv2} substituting Theorem \ref{prop:delicata2} with Theorem \ref{frobeniustrong}. 
\end{proof}

The following result shows that Theorem \ref{main:incorpus} is sharp, in the sense that the regimes of tradeoff between regularity of the surface and of the tangency set are optimal and not possible to improve. Such optimality is an immediate consequence of our constructions in Section \ref{sectioncounterexamples}. 

\begin{theorem}
    Suppose $V$ is a $k$-dimensional distribution of class $C^1$ and let $ q \in [1, \infty]$, $s, \alpha \in (0,1)$ be such that $s\in [0,1/2)$ and
\[
\alpha < 1 - \left(2 - \frac{1}{q} \right) s.
\]
Then, there exists an embedded $k$-dimensional submanifold $S$ of class $Y^{1+\alpha,q}$ and a Borel set $E$ with $\mathbb{1}_E\in W^{s,1}(S)$ such that 
$$E\subseteq \tau(S,V) \qquad\text{and}\qquad \Haus^k(E)>0.$$
\end{theorem}

\begin{proof}
 Proposition \ref{equivalenza} together with  Theorem \ref{lusinregSobolev} directly concludes the proof.
\end{proof}

    %
    %	BIBLIOGRAPHY
    %
    %
\printbibliography

    %
    %
    %	AFFILIATIONS
    %
    %
\vskip .5 cm
{\parindent = 0 pt\footnotesize
G.A.
\par
\smallskip
Dipartimento di Matematica,
Universit\`a di Pisa
\par
largo Pontecorvo 5,
56127 Pisa,
Italy
\par
\smallskip
e-mail: \texttt{giovanni.alberti@unipi.it}

\bigskip
A.Ma.
\par
\smallskip

Dipartimento di Matematica ``Tullio Levi-Civita'',
Universit\`a di Padova
\par
via Trieste 63, 
35121 Padova,
Italy
\par
\smallskip
e-mail: \texttt{annalisa.massaccesi@unipd.it}

\bigskip
A.Me.
\par
\smallskip
Universidad del Pais Vasco (UPV/EHU),
and\\
Ikerbasque, Basque Foundation for Science, Bilbao, Spain.
\par
Barrio Sarriena S/N 48940 Leioa, Spain.
\par
\smallskip
e-mail: \texttt{andrea.merlo@ehu.eus}
\par
}

\end{document}